\documentclass[a4paper,reqno,11pt]{amsart}
\usepackage{amsmath, amsbsy, amsthm, amssymb, amsfonts, latexsym, mathrsfs, color,enumerate}
\usepackage{stmaryrd}
\usepackage{bm}

\usepackage{graphicx,float,anysize}         
\usepackage[top=28mm,right=28mm,bottom=28mm,left=28mm]{geometry}

\definecolor{darkblue}{rgb}{0,0,0.8}

\newtheorem{theorem}{Theorem}[section]

\newtheorem{conjecture}[theorem]{Conjecture}
\newtheorem{lemma}[theorem]{Lemma}
\newtheorem{proposition}[theorem]{Proposition}
\newtheorem{corollary}[theorem]{Corollary}

\newtheorem{problem}[theorem]{Problem}

\theoremstyle{definition}

\newtheorem{remark}[theorem]{Remark}
\newtheorem{example}[theorem]{Example}

\newtheorem*{Conjecture}{Conjecture}
\newtheorem*{Definition}{Definition}

\DeclareMathOperator{\SL}{\mathrm{SL}}

\DeclareMathOperator{\Wr}{wr}

\DeclareMathOperator{\Aut}{Aut}

\DeclareMathOperator{\Sym}{Sym}
\DeclareMathOperator{\Alt}{Alt}

\newcommand{\la}{\langle}
\newcommand{\ra}{\rangle}

\renewcommand{\a}{\alpha}
\renewcommand{\b}{\beta}

\newcommand{\e}{\epsilon}

 \renewcommand{\L}{\Lambda}
\renewcommand{\l}{\lambda} \renewcommand{\O}{\Omega}

\makeatletter
\newcommand{\imod}[1]{\allowbreak\mkern4mu({\operator@font mod}\,\,#1)}
\makeatother

 \renewcommand{\to}{\rightarrow}

\newcommand{\leqs}{\leqslant}
\newcommand{\geqs}{\geqslant}
 
\newcommand{\what}{\widehat} 
 \newcommand{\vs}{\vspace{3mm}}

\begin{document}

\title{On the Saxl graph of a permutation group}

\author{Timothy C. Burness}
\address{T.C. Burness, School of Mathematics, University of Bristol, Bristol BS8 1TW, UK}
\email{t.burness@bristol.ac.uk}

\author{Michael Giudici}

\address{M. Giudici, Center for the Mathematics of Symmetry and Computation, Department of Mathematics and Statistics, The University of Western Australia, 35 Stirling Highway, Crawley WA 6009, Australia}
\email{michael.giudici@uwa.edu.au}

\date{\today}

\begin{abstract}
Let $G$ be a permutation group on a set $\Omega$. A subset of $\Omega$ is a base for $G$ if its pointwise stabiliser in $G$ is trivial. In this paper we introduce and study an associated graph $\Sigma(G)$, which we call the Saxl graph of $G$. The vertices of $\Sigma(G)$ are the points of $\Omega$, and two vertices are adjacent if they form a base for $G$. This graph encodes some interesting properties of the permutation group. We investigate the connectivity of $\Sigma(G)$ for a finite transitive group $G$, as well as its diameter, Hamiltonicity, clique and independence numbers, and we present several open problems. For instance, we conjecture that if $G$ is a primitive group with a base of size $2$, then the diameter of $\Sigma(G)$ is at most $2$. Using a probabilistic approach, we establish the conjecture for some families of almost simple groups. For example, the conjecture holds when $G=S_n$ or $A_n$ (with $n>12$) and the point stabiliser of $G$ is a primitive subgroup. In contrast, we can construct imprimitive groups whose Saxl graph is disconnected with arbitrarily many connected components, or connected with arbitrarily large diameter.
\end{abstract}

\maketitle

\section{Introduction}\label{s:intro}

Let $G \leqslant \Sym(\Omega)$ be a permutation group on a set $\Omega$. Recall that a subset of $\Omega$ is a \emph{base} for $G$ if its pointwise stabiliser in $G$ is trivial. The \emph{base size} of $G$, denoted by $b(G)$, is the minimal cardinality of a base for $G$. Clearly, $b(G)=1$ if and only if $G$ has a regular orbit (that is, $G$ has a trivial point stabiliser). Similarly, $b(G)=2$ if and only if a point stabiliser has a regular orbit on $\O$; in this situation, we say that $G$ is a \emph{base-two} group. 

Bases for finite permutation groups have been studied indirectly since the early days of group theory in the nineteenth century, finding a wide range of applications. For example, if $|\O|=n$ then the definition of a base immediately implies that $|G| \leqs n^{b(G)}$ and thus upper bounds on $b(G)$ can be used to bound the order of $G$ in terms of its degree (see \cite{Babai}, for example). More recently, bases have been used extensively in the computational study of permutation groups (see \cite[Chapter 4]{Seress} for more details). We refer the reader to the excellent survey article \cite{BC} for more background on bases and their connections to other invariants of groups and graphs.

In recent years, there have been significant advances in our understanding of base sizes of primitive permutation groups. Much of this interest stems from an influential conjecture of Cameron and Kantor \cite{CK}, which asserts that there is an absolute constant $c$ such that $b(G) \leqs c$ for all so-called \emph{non-standard} finite almost simple primitive groups $G$. Roughly speaking, an almost simple primitive group $G \leqs {\rm Sym}(\O)$ with socle $T$ is \emph{standard} if $T=A_n$ and $\O$ is a set of subsets or partitions of $\{1, \ldots, n\}$, or $T$ is a classical group and $\O$ is a set of subspaces of the natural module for $T$ (otherwise, $G$ is \emph{non-standard}). This conjecture was proved by Liebeck and Shalev \cite{LSh99} using probabilistic methods. In later work \cite{Bur7, BGS1, BLS, BOW}, it was shown that $c=7$ is the best possible constant. 

In fact, it is possible to show that $b(G)=2$ for many non-standard groups $G$ and it is natural to seek a complete classification. This ambitious project was initiated by Jan Saxl in the broader context of finite primitive permutation groups and we briefly recall some of the main results. By combining the main theorems in \cite{BGS1,James2006}, one obtains a complete classification of the base-two primitive actions of symmetric and alternating groups. There is an analogous result for almost simple primitive groups with a sporadic socle in \cite{BOW,orb} and partial results in \cite{BGS, BGS2, James} for classical groups. For more general primitive groups, see \cite{Fawcett1, Fawcett2} for results on diagonal-type and twisted wreath product groups. In the context of a primitive affine group $G = VH$, where $H \leqs {\rm GL}(V)$ is irreducible, we get $b(G)=2$ if and only if $H$ has a regular orbit on $V$. In this setting, the base-two problem can be viewed as a natural problem in representation theory (indeed, it is closely related to the famous $k(GV)$-problem \cite{Schmid}, which establishes part of a conjecture of Brauer on defect groups of blocks). For example, see \cite{FMOW,FOS} for some recent results on base-two affine groups with a quasisimple point stabiliser.

\vs

We introduce the following graph associated to a base-two permutation group.

\begin{Definition}
Let $G \leqs {\rm Sym}(\O)$ be a permutation group on a set $\O$. The \emph{Saxl graph} of $G$, denoted by $\Sigma(G)$, has vertex set $\O$ and two vertices are adjacent if and only if they form a base for $G$.
\end{Definition}

The main goal of this paper is to investigate the Saxl graphs of finite transitive permutation groups $G$. In Section \ref{s:obs}, we start by recording some elementary properties of $\Sigma(G)$ and we present several examples. Next, in Sections \ref{s:val} and \ref{s:con} we study the valency and connectivity of $\Sigma(G)$. For example, in Proposition \ref{p:prime} we classify the transitive groups such that $\Sigma(G)$ has prime valency, and for primitive groups we address the existence of Eulerian cycles in Propositions \ref{p:eul} and \ref{p:eulprim}. In Section \ref{ss:prob}, we develop a probabilistic approach to compute bounds on the valency of $\Sigma(G)$ (the main result is Lemma \ref{l:key}) and we study some related asymptotics in Section \ref{ss:asymp}. 

It is easy to see that the Saxl graph of a base-two primitive permutation group is connected. In fact, it appears that these graphs are highly connected and we propose the following striking conjecture (see Conjecture \ref{c:diam2}).

\begin{Conjecture}
\emph{Let $G$ be a finite primitive permutation group with $b(G)=2$ and Saxl graph $\Sigma(G)$. Then any two vertices in $\Sigma(G)$ have a common neighbour.}
\end{Conjecture}

Notice that this implies that $\Sigma(G)$ has diameter at most $2$ and that every edge in $\Sigma(G)$ is contained in a triangle. We present some evidence for the conjecture at the end of Section \ref{s:con}, which is extended in Sections \ref{s:sym} and \ref{s:spor}, where we use probabilistic methods to verify the conjecture for various families of almost simple groups with socle an alternating or sporadic group. It is easy to see that the primitivity hypothesis is essential -- in fact, we can construct imprimitive base-two groups $G$ such that $\Sigma(G)$ is either disconnected with arbitrarily many connected components, or connected with arbitrarily large diameter (see Propositions \ref{p:c2} and \ref{p:c3}). 

Finally, we refer the reader to Section \ref{ss:op} for a number of open problems.

\bigskip

\noindent \textbf{Acknowledgements.} Both authors thank Peter Cameron, Robert Guralnick, Martin Liebeck and Eamonn O'Brien for helpful comments. They also thank an anonymous referee for detailed comments and suggestions that have significantly improved the clarity of the paper. The second author is supported by ARC Discovery Project DP160102323. 

\section{First observations}\label{s:obs}

Let $G \leqslant \Sym(\Omega)$ be a transitive permutation group of degree $n$ with point stabiliser $H$ and Saxl graph $\Sigma(G)$. In this preliminary section, we discuss some of the elementary properties of $\Sigma(G)$. Notice that if $G$ is regular, then $\Sigma(G)$ is a complete graph, so let us assume otherwise. Therefore, $H \ne 1$ and $b(G) \geqs 2$. Of course, if $b(G)>2$ then $\Sigma(G)$ is empty, so we will also assume that $G$ has a base of size $2$.  

\begin{lemma}\label{l:1}
Let $G \leqslant \Sym(\Omega)$ be a finite transitive permutation group with point stabiliser $H$. Suppose $b(G)=2$ and let $\Sigma(G)$ be the Saxl graph of $G$.
\begin{itemize}\addtolength{\itemsep}{0.2\baselineskip}
\item[{\rm (i)}] $G \leqs {\rm Aut}(\Sigma(G))$ acts transitively on the set of vertices of $\Sigma(G)$. In particular, $\Sigma(G)$ is a vertex-transitive graph with no isolated vertices.
\item[{\rm (ii)}] $\Sigma(G)$ is connected if $G$ is primitive.
\item[{\rm (iii)}] $\Sigma(G)$ is complete if and only if $G$ is Frobenius. 
\item[{\rm (iv)}] $\Sigma(G)$ is arc-transitive if $G$ is $2$-transitive, and edge-transitive if $G$ is $2$-homogeneous.   
\item[{\rm (v)}] $\Sigma(G)$ has valency $r|H|$, where $r$ is the number of regular orbits of $H$ on $\O$. 
\item[{\rm (vi)}] $\Sigma(G)$ is a subgraph of $\Sigma(K)$ for every subgroup $K$ of $G$.
\item[{\rm (vii)}] $G$ acts semiregularly on the set of arcs of $\Sigma(G)$.
\item[{\rm (viii)}] If $G_\alpha=G_\beta$, then $\alpha$ and $\beta$ have the same set of neighbours in $\Sigma(G)$.
\end{itemize}
\end{lemma}

\begin{proof}
First observe that if $\{\a,\b\}$ is a base for $G$, then so is $\{\a^g,\b^g\}$ for all $g \in G$. This gives part (i), and (ii) follows from the fact that the connected components of $\Sigma(G)$ form a system of imprimitivity for $G$. Part (iii) is also clear: $\Sigma(G)$ is complete if and only if every two-point stabiliser is trivial, which is the defining property of a Frobenius group (since $H \ne 1$). In particular, note that $\Sigma(G)$ is complete if $G$ is $2$-transitive (here the base-two condition implies that $G$ is sharply $2$-transitive). Part (iv) is immediate from the definitions (note that if $G$ is sharply $2$-homogeneous, then $G$ acts transitively on edges, but not on arcs). Since $\{\alpha,\beta\}$ is a base for $G$ if and only if $G_\alpha$ acts regularly on $\beta^{G_\alpha}$, it follows that the set of neighbours of $\alpha$ in $\Sigma(G)$ is the union of the regular orbits of $G_\alpha$. This gives part (v). For (vi), observe that if $\{\alpha,\beta\}$ is a base for $G$ then $K_{\alpha,\beta}\leqslant G_{\alpha,\beta}=1$ and so $\{\alpha,\beta\}$ is a base for $K$. In particular, each edge of $\Sigma(G)$ is an edge of $\Sigma(K)$. Part (vii) is clear because arcs in $\Sigma(G)$ correspond to ordered bases for $G$, so the stabiliser of an arc is trivial. Finally, if $G_\alpha=G_\beta$ then $G_{\alpha,\gamma}=G_{\beta,\gamma}$ for all $\gamma\in\Omega$ and so (viii) follows.
\end{proof}

\begin{remark}\label{r:og}
Fix a point $\alpha \in\Omega$. Recall that the orbits of $G$ on $\Omega\times \Omega$ are called \emph{orbitals} and each orbital defines the set of arcs of an \emph{orbital digraph} for $G$ with vertex set $\Omega$. For each orbital $\Delta$, the set $\Delta(\alpha)=\{\beta\in\Omega \,: \, (\alpha,\beta)\in\Delta\}$ is an orbit of $G_\alpha$, which we call a \emph{suborbit} of $G$. This gives a one-to-one correspondence between orbitals and suborbits. As noted in the proof of Lemma \ref{l:1}(v), the union of the regular orbits of $G_\alpha$ form the set of neighbours of $\alpha$ in $\Sigma(G)$. Each regular suborbit gives an orbital of size $|\Omega||G_\alpha|$ and the arcs of $\Sigma(G)$ are the elements of these orbitals. In particular, $\Sigma(G)$ is the \emph{generalised orbital graph} corresponding to the regular suborbits of $G$.
\end{remark}

\begin{example}\label{ex:1}
Let $q$ be a prime power and consider the natural action of $G = {\rm GL}_{2}(q)$ on the set $\O$ of nonzero vectors in the natural module $\mathbb{F}_q^2$. Here two vectors are adjacent in the Saxl graph $\Sigma(G)$ if and only if they are linearly independent, so $\Sigma(G)$ is the complete multipartite graph with $q+1$ parts of size $q-1$. 
\end{example}

Many other permutation groups have a complete multipartite Saxl graph: 

\begin{itemize}\addtolength{\itemsep}{0.2\baselineskip}
\item[{\rm (a)}] Let $L$ be a group of order $n$ acting regularly on itself, let $K$ be a Frobenius group on a set $\Delta$ of size $m$ and consider the imprimitive action of $G=L\times K$ on the set $\Omega=L\times\Delta$. Since $G_{(\ell_1,\delta),(\ell_2,\delta)}=K_\delta\neq 1$ and  $G_{(\ell_1,\delta_1),(\ell_2,\delta_2)}=K_{\delta_1,\delta_2}=1$ (for $\delta_1 \ne \delta_2$), it follows that $\Sigma(G)$ is a complete multipartite graph with $m$ parts of size $n$.
\item[{\rm (b)}] Let $p$ be a prime and consider the action of $G={\rm AGL}_{1}(p^f) = (C_p)^f{:}C_{p^f-1}$ on the set of cosets of $C_{(p^f-1)/r}$, where $r>1$ is a proper divisor of $p^f-1$. Then $G$ is imprimitive, with a system of imprimitivity comprising $p^f$ blocks of size $r$, and  $G$ acts sharply $2$-transitively on the set of blocks. Since elements in the same block have the same stabiliser, we deduce that $\Sigma(G)$ is a complete multipartite graph with $p^f$ parts of size $r$.
\item[{\rm (c)}] Take the action of $G={\rm GL}_n(p)$ on the set of cosets of a subgroup 
$(C_p)^{n-1}$, which is the socle of the stabiliser in $G$ of a $1$-dimensional subspace of the natural module $\mathbb{F}_p^n$ (here $p$ is a prime). Then $G$ is imprimitive, with a system of imprimitivity comprising $m=(p^n-1)/(p-1)$ blocks of size $\ell=(p-1)|{\rm GL}_{n-1}(p)|$.  Since the stabiliser of a point is normal in the stabiliser of a block, all points in the same block have the same stabiliser, and so $\Sigma(G)$ is multipartite with $m$ parts of size $\ell$.  Moreover, the socles of the stabilisers of distinct $1$-spaces intersect trivially and so $\Sigma(G)$ is complete multipartite.
\end{itemize}

As the next two examples demonstrate, other familiar graphs can also arise as Saxl graphs.

\begin{example}\label{ex:payley}
Set $V = \mathbb{F}_{q^2}$, where $q$ is odd, and fix $\xi \in V$ of order $(q+1)/2$ in the multiplicative group $\mathbb{F}_{q^2}^{\times}$. Notice that $\langle \xi\rangle$ acts on $V$ by right multiplication, with $2(q-1)$ regular orbits. Let $\sigma:V\rightarrow V$ be the Frobenius map $x^\sigma=x^q$ and note that $\xi^\sigma=\xi^q=\xi^{-1}$, so 
\[
H:=\langle \xi,\sigma\rangle\cong D_{q+1}.
\] 
Moreover, $\sigma$ permutes the orbits of $\langle \xi\rangle$. Now $\sigma$ fixes the orbit containing $x$ if and only if $x^q=x\xi^i$ for some $i$, which holds if and only if $x^{q-1}\in \langle \xi\rangle$. Thus the orbits of $H$ on $V$ of size $(q+1)/2$ consist precisely of the elements of $V$ lying in the subgroup of $\mathbb{F}_{q^{2}}^{\times}$ of order $(q^2-1)/2$, that is, the nonzero squares. In particular, the non-squares in $V$ lie in regular $H$-orbits. Therefore, if we consider the affine group $G=V{:}H < {\rm A\Gamma L}(V)$, then $\Sigma(G)$ has valency $(q^2-1)/2$ and two elements are adjacent if and only if their difference is a non-square. If $\lambda\in V$ is a non-square, then the map $x\mapsto \lambda x$ gives an isomorphism from $\Sigma(G)$ to its complement, which is the familiar \emph{Paley graph} $P(q^2)$ on $V$, where two elements are adjacent if and only if their difference is a square. In particular, $\Sigma(G)$ is self-complementary. 
\end{example}

\begin{example}\label{ex:johnson}
Let $G={\rm PGL}_2(q)$ and note that $G$ acts $3$-transitively on the set $\Delta$ of 
$1$-dimensional subspaces of $\mathbb{F}_{q}^2$. Let us view $G$ as a permutation group on  
the set $\Omega$ of $2$-subsets of $\Delta$ and fix $\alpha=\{\langle e_1\rangle,\langle e_2\rangle\}\in\Omega$, so $G_\alpha = D_{2(q-1)}$. Notice that the $3$-transitivity of $G$ on $\Delta$ implies that if $\b = \{\langle u\rangle ,\langle v\rangle\}\in \Omega$, then  
$$\Omega_{\langle u\rangle ,\langle v\rangle} := {\big\{}\{\langle u\rangle ,\langle w\rangle\},\{\langle v\rangle ,\langle w\rangle\} \, : \, \langle w\rangle \neq \langle u\rangle, \langle v\rangle{\big\}}$$
is an orbit of $G_{\b}$ of size $2(q-1)$. In particular, $G_{\a}$ has a regular orbit on $\O$ and thus $b(G)=2$. Moreover, if $\l_1,\l_2 \in \mathbb{F}_q^{\times}$ and $\l_1 \ne \l_2$, then the element of $G_{\alpha}$ with preimage 
$$\begin{pmatrix}
0&\lambda_1\\
\lambda_2^{-1} &0
\end{pmatrix}$$
fixes $\{\langle e_1+\lambda_1e_2\rangle,\langle e_1+\lambda_2 e_2\rangle\} \in \O$ and we deduce that $G_\alpha$ has a unique regular suborbit. It follows that two vertices in the corresponding Saxl graph $\Sigma(G)$ are adjacent if and only if they have a point in common. Therefore, $\Sigma(G)$ is the \emph{Johnson graph} $J(q+1,2)$, which is also known as the \emph{triangular graph} $T(q+1)$. 
\end{example}

\begin{remark}
Notice that in all of the above examples, the given Saxl graphs are \emph{strongly regular}, that is, all vertices have the same valency and there are constants $\lambda$ and $\mu$ such that every pair of adjacent (respectively, non-adjacent) vertices have $\lambda$ (respectively, $\mu$) common neighbours. However, this is not a general property of Saxl graphs. For example, the graphs constructed in Proposition \ref{p:c3} (see Section \ref{s:con}) are not strongly regular.
\end{remark}

It is natural to study the behaviour of Saxl graphs under some standard constructions. With this in mind, recall that if $\Gamma$ and $\Sigma$ are graphs with vertex sets $V\Gamma$ and $V\Sigma$ respectively, then their \emph{direct product} $\Gamma\times\Sigma$ is the graph with vertex set $V\Gamma\times V\Sigma$ and $\{(u_1,v_1),(u_2,v_2)\}$ is an edge if and only if $\{u_1,u_2\}$ and $\{v_1,v_2\}$ are edges in $\Gamma$ and $\Sigma$, respectively. 

\begin{lemma}\label{l:prod}
Suppose $G \leqs {\rm Sym}(\O)$ and $K \leqs {\rm Sym}(\Delta)$ are both base-two groups and consider the natural action of $G \times K$ on $\Omega\times\Delta$. Then $b(G \times K)=2$ and $\Sigma(G\times K)=\Sigma(G)\times\Sigma(K)$.
\end{lemma}

\begin{proof}
Let $(\omega_1,\delta_1),(\omega_2,\delta_2)\in\Omega\times\Delta$ and note that $(G\times K)_{(\omega_1,\delta_1),(\omega_2,\delta_2)}=G_{\omega_1,\omega_2}\times K_{\delta_1,\delta_2}$. Therefore, $\{(\omega_1,\delta_1),(\omega_2,\delta_2)\}$ is a base for $G\times K$ if and only if $\{\omega_1,\omega_2\}$ and $\{\delta_1,\delta_2\}$ are bases for $G$ and $K$, respectively. Moreover, $\{(\omega_1,\delta_1),(\omega_2,\delta_2)\}$ is an edge of $\Sigma(G\times K)$ if and only if $\{\omega_1,\omega_2\}$ is an edge of $\Sigma(G)$ and $\{\delta_1,\delta_2\}$ is an edge of $\Sigma(K)$. We conclude that $\Sigma(G\times K)=\Sigma(G)\times\Sigma(K)$.
\end{proof}

Given permutation groups $L \leqslant \Sym(\Delta)$ and $P\leqslant S_k$, we can consider the product action of the wreath product $G=L \Wr P$ on $\Omega=\Delta^k$. Bases for $G$ were investigated in \cite[Theorem 2.13]{BC} and by inspecting the proof we obtain the following result. Here the stabiliser in $P$ of a partition $\mathcal{P} = \{X_1, \ldots, X_t\}$ of $\{1, \ldots, k\}$ is defined to be $\bigcap_{i}P_{X_i}$, where $P_{X_i}$ is the setwise stabiliser of $X_i$ in $P$.

\begin{lemma}\label{lem:findbase}
Let $\omega=(\a_1,\ldots,\a_k)$ and $\omega'=(\b_1,\ldots,\b_k)$ be elements of $\Omega=\Delta^k$ and let $\mathcal{P}$ be the partition of $\{1,\ldots,k\}$ with the property that $i$ and $j$ lie in the same part if and only if $(\a_i,\b_i)^x = (\a_j,\b_j)$ for some $x \in L$. Then $\{\omega,\omega'\}$ is a base for $L \Wr P$ if and only if each $\{\a_i,\b_i\}$ is a base for $L$ and the stabiliser of $\mathcal{P}$ in $P$ is trivial.
\end{lemma}

Recall that the \emph{distinguishing number} $D(G)$ of a permutation group $G \leqs {\rm Sym}(\O)$ is the smallest size of a partition of $\O$ whose stabiliser in $G$ is trivial. For example, $D({\rm Sym}(\O)) = |\O|$ and $D({\rm Alt}(\O)) = |\O|-1$.

\begin{corollary}\label{cor:base2}
If $b(L)=2$ then $L \Wr P$ has a base of size $2$ on $\Omega = \Delta^k$ if and only if the number of orbits of $L$ on ordered bases of size $2$ is at least the distinguishing number of $P$.
\end{corollary}

This corollary shows that $\Sigma(L)$ and $\Sigma(L \Wr P)$ can have very different properties.  For example, suppose $L \leqs {\rm Sym}(\Delta)$ is $2$-transitive and $b(L)=2$, in which case $\Sigma(L)$ is complete. If we take a non-trivial subgroup $P \leqs S_k$, then Corollary \ref{cor:base2} implies that $b(L \Wr P) \geqs 3$ and thus $\Sigma(L \Wr P)$ is empty. 

\vs

These preliminary observations suggest a number of natural problems, some of which we will investigate in this paper. For instance, it is interesting to seek conditions on $G$ that are sufficient for the connectedness of $\Sigma(G)$ (as noted in part (ii) of Lemma \ref{l:1}, primitivity is one such condition). Moreover, if $\Sigma(G)$ is connected then we can  investigate its diameter, as well as the existence of Hamiltonian and Eulerian cycles. It is also natural to explore the connectivity properties of $\Sigma(G)$ for imprimitive base-two groups; for example, can we construct groups $G$ such that $\Sigma(G)$ has arbitrarily many connected components? Some familiar graph-theoretic invariants, such as the clique and independence numbers, also have a natural group-theoretic interpretation in the context of Saxl graphs.

\section{Valency}\label{s:val}

Let $G \leqslant \Sym(\Omega)$ be a transitive permutation group of degree $n$ with point stabiliser $H$ and Saxl graph $\Sigma(G)$. Assume that $b(G)=2$.
Recall that $\Sigma(G)$ is vertex-transitive; we write ${\rm val}(G)$ for its valency. 
By Lemma \ref{l:1}(v), ${\rm val}(G) = r|H|$, where $r$ is the number of regular orbits of $H$ on $\O$. In this section, we investigate the valency of $\Sigma(G)$ in more detail. 

\subsection{Prime valency}\label{ss:pv}

We start by classifying the transitive groups $G$ such that $\Sigma(G)$ has prime valency.

\begin{proposition}\label{p:prime}
Let $G \leqslant \Sym(\Omega)$ be a finite transitive permutation group with $b(G)=2$. Then $\Sigma(G)$ has prime valency $p$ if and only if $G$ is one of the following:
\begin{itemize}\addtolength{\itemsep}{0.2\baselineskip}
\item[{\rm (i)}] $G=C_p\Wr C_2$ acting imprimitively on $2p$ points and $\Sigma(G)$ is the complete bipartite graph $K_{p,p}$.
\item[{\rm (ii)}] $G=S_3$, $p=2$ and $\Sigma(G)$ is the complete graph $K_3$.
\item[{\rm (iii)}] $G={\rm AGL}_{1}(2^f)$, where $p=2^f-1$ is a Mersenne prime and $\Sigma(G)=K_{p+1}$.
\end{itemize}
\end{proposition}

\begin{proof}
First observe that if $G$ is one of the groups in part (i), (ii) or (iii), then $\Sigma(G)$ has the given form and thus ${\rm val}(G)$ is a prime. For the remainder, let us assume $G \leqslant \Sym(\Omega)$ is transitive, $b(G)=2$ and ${\rm val}(G)=p$ is a prime. 

Let $\alpha\in\Omega$. The set of neighbours of $\alpha$ in $\Sigma(G)$ is the union of the regular orbits of $G_\alpha$, so $|G_\alpha|=p$ and $G_\alpha$ has a unique regular orbit $\Delta$. Moreover, $G_{\a}$ fixes every point in $\Omega\setminus\Delta$. The set of fixed points of $G_\alpha$ forms a block of imprimitivity for $G$, which gives rise to a system of imprimitivity $\mathcal{P}$. The complement of a block is a union of blocks in $\mathcal{P}$ and so $\Delta$ is a union of blocks. Since $|\Delta|=p$, either $\Delta$ is also a block in $\mathcal{P}$ and $|\Omega|=2p$, or $G_{\alpha}$ has a unique fixed point and $|\Omega|=p+1$.

In the first case, $G\leqslant S_p\Wr C_2$ acting imprimitively on $2p$ points.  Moreover, $|G|=2p|G_\alpha|=2p^2$. If $p=2$ then $G = C_2 \Wr C_2$ and (i) holds. Now assume $p>2$. By Sylow's Theorem, $G$ has a unique Sylow $p$-subgroup $P$ and from the structure of $S_p\Wr C_2$ we deduce that $P=C_p\times C_p$. Since $G$ is transitive, it contains an element interchanging the two blocks of imprimitivity and so $G=C_p\Wr C_2$. The structure of $\Sigma(G)$ in this case is transparent.

If $G_\alpha$ has a unique fixed point then $G$ is sharply 2-transitive of degree $p+1$. By a theorem of Zassenhaus \cite{Z}, $G$ has a regular normal elementary abelian subgroup $N = (C_r)^f$ for some prime $r$ and $p=r^f-1$. Thus either $r^f=3$ and $p=2$, or $r=2$ and $p$ is a Mersenne prime. This gives the examples in parts (ii) and (iii) of the proposition. 
\end{proof}

\subsection{Even valency}\label{ss:even}

Recall that an \emph{Eulerian cycle} in a graph is a cycle that uses each edge exactly once; if such a cycle exists, then the graph is said to be \emph{Eulerian}. A famous result due to Euler himself states that a connected graph is Eulerian if and only if the degree of every vertex is even. In particular, a connected Saxl graph is Eulerian if and only if its valency is even. 

Our first result concerns the existence of Eulerian cycles in the Saxl graphs of almost simple primitive groups.

\begin{proposition}\label{p:eul}
Let $G \leqs {\rm Sym}(\O)$ be an almost simple primitive permutation group with socle $T$, point stabiliser $H$ and base size $2$. Then one of the following holds:
\begin{itemize}\addtolength{\itemsep}{0.2\baselineskip}
\item[{\rm (i)}] The Saxl graph of $G$ is Eulerian.
\item[{\rm (ii)}] $(G,H) = ({\rm M}_{23},23{:}11)$ and $\Sigma(G)$ is non-Eulerian.
\item[{\rm (iii)}] $G = A_p$ and $H = {\rm AGL}_{1}(p) \cap G = C_p{:}C_{(p-1)/2}$, where $p$ is a prime such that $p \equiv 3 \imod{4}$ and $(p-1)/2$ is composite.
\item[{\rm (iv)}] $T = {\rm L}_{n}^{\e}(q)$, $H \cap T = C_a{:}C_n$ and $G \ne T$, where $a = \frac{q^n-\e}{(q-\e)(n,q-\e)}$, $n \geqs 3$ is a prime and $T \ne {\rm U}_{3}(3), {\rm U}_{5}(2)$.
\end{itemize}
\end{proposition}

\begin{proof}
Let $r$ be the number of regular orbits of $H=G_{\a}$ on $\O$. By primitivity, $\Sigma(G)$ is connected, so $\Sigma(G)$ is non-Eulerian if and only if both $r$ and $|H|$ are odd. The condition that $|H|$ is odd is highly restrictive and the possibilities for $G$ and $H$ are known. Indeed, \cite[Theorem 2]{LS91} reduces the problem to the cases recorded in Table \ref{tab:odd} (we can exclude the two possibilities in \cite[Theorem 2]{LS91} with $T$ the Monster group since the relevant subgroups $59{:}29$ and $71{:}35$ are not maximal; see \cite[Section 3.6]{Wil} for further details). Since $b(G)=2$, we can immediately eliminate the case in the first row of Table \ref{tab:odd} (indeed, in this case $G$ is $2$-transitive, but not sharply $2$-transitive).

\begin{table}
\[
\begin{array}{lll} \hline
T & H \cap T & \mbox{Conditions} \\ \hline
{\rm L}_{2}(q),\, q=p^f & (C_p)^f{:}C_{(p^f-1)/2} & \mbox{$p$ prime, $p^f \equiv 3 \imod{4}$} \\
A_p & C_p{:}C_{(p-1)/2} & \mbox{$p$ prime, $p \equiv 3 \imod{4}$, $p \ne 7,11,23$}  \\
{\rm L}_{n}^{\e}(q),\, \e = \pm  & C_{a}{:}C_n & \mbox{$a = \frac{q^n-\e}{(q-\e)(n,q-\e)}$, $n \geqs 3$ prime, $T \ne {\rm U}_{3}(3), {\rm U}_{5}(2)$} \\
{\rm M}_{23}, {\rm Th}, \mathbb{B} & 23{:}11, 31{:}15, 47{:}23 \, \mbox{(resp.)} & \\ \hline
\end{array}
\] 
\caption{Primitive almost simple groups with odd order point stabiliser}
\label{tab:odd}
\end{table}

Next suppose $T=A_p$ and $H \cap T = C_p{:}C_{(p-1)/2}$ as in the second row. Note that if 
$G = S_p$ then $H = {\rm AGL}_{1}(p)$ has even order, so we may assume that $G = A_p$ and $H = C_p{:}C_{(p-1)/2}$, which is the normaliser of a Sylow $p$-subgroup of $G$. Set $q=(p-1)/2$ and assume $q$ is a prime. If $g \in G$ and $|H \cap H^g|=p$ then $g$ normalises the unique Sylow $p$-subgroup of $H$ and thus $g \in H$, which is a contradiction. Therefore, 
$|H \cap H^g| \in \{1,q\}$ for all $g \in G \setminus H$. We need to show that $r$ is even. To see this, first observe that $H$ contains $p(q-1)$ elements of order $q$, each of which has cycle-shape $(q^2,1)$ as an element of $G$. Therefore, if $y \in H$ is such an element, then $|C_G(y)|=q^2$ and $|y^G \cap H|=p(q-1)$, which implies that $y$ has 
\[
\frac{|y^G \cap H||C_G(y)|}{|H|} = q(q-1)
\]
fixed points on $\O$. Each fixed point is contained in a unique $H$-orbit, so $H$ has $q^2-q-1$ orbits of length $p$ and thus 
\[
r = \frac{|G:H|-p(q^2-q-1)-1}{|H|} 
\]
is even, as required.

Next assume $G = {\rm L}_{n}^{\e}(q)$ and $H=K{:}L$, where $K=C_{a}$, $L=C_n$ and $a,n$ are described in Table \ref{tab:odd}. First we claim that $(a,n)=1$. If $n$ divides $q-\e$ then the claim follows from \cite[Lemma A.4]{BG_book}, so we may assume $(n,q-\e)=1$. The claim is clear if $q$ is an $n$-power, so let us assume $(n,q)=1$, in which case $n$ divides $q^{n-1}-1$ by Fermat's Little Theorem. If $n$ also divides $a$, then $n$ divides $q^n-\e$, so $n$ divides $q^{n-1}(q-\e)$, which is a contradiction. We conclude that $(a,n)=1$. 

Next observe that $H=N_G(J)$ for every non-trivial subgroup $J \leqslant K$. Now, if $s$ is a prime divisor of $a$, then $H$ has a unique subgroup $J$ of order $s$ and this is contained in $K$. Therefore, if $g \in G$ and $J \leqslant H\cap H^g$, then $g \in N_G(J)$ and thus $g \in H$. It follows that $|H\cap H^g| \in \{1,n\}$ for all $g \in G \setminus H$, so every non-trivial, non-regular orbit of $H$ has length $a$.

Let $y \in H$ be an element of order $n$. By \cite[Lemma 5.2(i)]{LSh15} we have $|y^G \cap H|=a(n-1)$, so every element of $H \setminus K$ has order $n$. Since $\la y \ra$ is a Sylow $n$-subgroup of $H$, we see that $H$ contains $k(n-1)$ elements of order $n$, where $k$ is the number of Sylow $n$-subgroups of $H$. Therefore $k=a$ and thus $N_H(\la y \ra) = \la y \ra$ (we also deduce that $n$ divides $a-1$). In particular, each $H$-orbit of length $a$ contains a unique fixed point of $y$ and thus
\[
r = \frac{|G:H| - a(b-1)-1}{|H|},
\]
where 
\[
b = \frac{|y^G \cap H||C_G(y)|}{|H|}
\]
is the number of fixed points of $y$ on $\O$. Since $b$ is even we conclude that $r$ is even
and thus $\Sigma(G)$ is Eulerian. 

Finally, let us consider the three cases in the final row of Table \ref{tab:odd}. If $(G,H) = ({\rm M}_{23},23{:}11)$ then we can use {\sc Magma} \cite{magma} to show that $r=159$ and thus $\Sigma(G)$ is non-Eulerian, giving part (ii) in the proposition. 

Next assume $(G,H) = (\mathbb{B},47{:}23)$, so $|H \cap H^g| \in \{1,23\}$ for all $g \in G \setminus H$. Fix an element $y \in H$ of order $23$ and note that $|C_G(y)|=46$ and $\la y \ra$ meets the two $G$-classes of elements of order $23$. We can compute $|y^G \cap H|$ from the stored fusion information in the \textsf{GAP} Character Table Library \cite{CTblLib} and we quickly deduce that $y$ has $22$ fixed points on $\O$, say $\a, \a_1, \ldots, \a_{21}$. Now each $\a_i$ corresponds to an $H$-orbit of length $47$ and we note that $y$ has a unique fixed point on each of these orbits, so $H$ has exactly $21$ orbits of length $47$ and thus 
\[
r = \frac{|G:H| - 21\cdot 47 -1}{|H|} = 3555468205105618098822038852
\]
is even. Therefore, $\Sigma(G)$ is Eulerian.

Finally, suppose $(G,H) = ({\rm Th}, 31{:}15)$, so $|H \cap H^g| \in \{1,3,5,15\}$ for all $g \in G \setminus H$. Suppose $u \in H$ has order $15$. By arguing as in the previous case, we calculate that $u$ has $8$ fixed points on $\O$, one of which is $\a$. Moreover, $u$ has a unique fixed point in each $H$-orbit of length $31$, whence $H$ has $7$ orbits of length $31$. Next assume $v \in H$ has order $5$. Here we calculate that $v$ has precisely $800$ fixed points on $\O$. Moreover, $v$ has a unique fixed point on each 
$H$-orbit of length $31$ and $3$ fixed points on each $H$-orbit of length $93$ (since $v$ has cycle-shape $(5^{18},1^3)$ on $H/\la v \ra$). This implies that $H$ has precisely $792/3 = 264$ orbits of length $93$. Finally, suppose $w \in H$ has order $3$. Here $w$ has $23328$ fixed points on $\O$, with a unique fixed point on each $H$-orbit of length $31$ and $5$ fixed points on the $H$-orbits of length $155$. Therefore, $H$ has $23320/5=4664$ orbits of length $155$ and thus
\[
r = \frac{|G:H| - 7\cdot 31 - 264\cdot 93 - 4664\cdot 155 -1}{|H|} = 419682939254
\]
is even. Once again we conclude that $\Sigma(G)$ is Eulerian.
\end{proof}

\begin{remark}
It would be interesting to determine the parity of $r$ in the cases left open in parts (iii) and (iv) of 
Proposition \ref{p:eul}. However, the analysis of the subdegrees is rather more complicated in these cases and we do not pursue this here.
\end{remark}

As an easy application of the O'Nan-Scott theorem (see \cite[Theorem 4.1A]{DM}), we can state a more general result for all finite primitive permutation groups. 

\begin{proposition}\label{p:eulprim}
Let $G \leqs {\rm Sym}(\O)$ be a finite primitive permutation group with socle $T^k$, point stabiliser $H$ and base size $2$. Then one of the following holds:
\begin{itemize}\addtolength{\itemsep}{0.2\baselineskip}
\item[{\rm (i)}] The Saxl graph of $G$ is Eulerian.
\item[{\rm (ii)}] $G$ is soluble.
\item[{\rm (iii)}] $G \leqs L \Wr P$ acting with its product action on $\O = \Delta^k$, where $k \geqs 1$, $L \leqs {\rm Sym}(\Delta)$ is one of the almost simple groups in Table \ref{tab:odd} with socle $T$, and $P \leqs S_k$ is a transitive group of odd order. 
\end{itemize}
\end{proposition}

By Proposition \ref{p:eul}, $(G,H) = ({\rm M}_{23},23{:}11)$ is an example of a base-two primitive group with a non-Eulerian Saxl graph. In addition, we also have the prime valency examples presented in part (iii) of Proposition \ref{p:prime}, but we do not know if there are infinitely many groups with this property (of course, there will be infinitely many if there are infinitely many Mersenne primes). With this in mind, it would be interesting to study the valency of the Saxl graphs of wreath products in product action in more detail. The following result shows that the parity of ${\rm val}(L)$ need not be preserved under the wreath product construction $L \Wr P$, even when the top group $P \leqs S_k$ has odd order.

\begin{proposition}
Let $L\leqslant \Sym(\Delta)$ be a transitive base-two group such that for $\delta\in\Delta$, $L_\delta$ has $k$ regular orbits for some $k\geqs 2$. Let $p$ be a prime and set $G=L\Wr C_p$ with its product action on $\Omega=\Delta^p$. Then $b(G)=2$ and $\Sigma(G)$ has valency $|L_{\delta}|^p(k^p-k)$.
\end{proposition}

\begin{proof}
First observe that the distinguishing number of $C_p$ is $2$ (in its regular action on $p$ points) and so Corollary \ref{cor:base2} implies that $b(G)=2$. In fact, the stabiliser in $C_p$ of any non-trivial partition of $\{1,\ldots,p\}$ is trivial and so Lemma \ref{lem:findbase} implies that the neighbours of $(\delta,\ldots,\delta)$ in $\Sigma(G)$ are precisely the $p$-tuples whose entries lie in regular $L_\delta$-orbits and at least two of the entries come from distinct regular $L_\delta$-orbits. Therefore 
\begin{align*}
{\rm val}(G) &=\left(\begin{array}{l} \# \, \textrm{$p$-tuples with entries from} \\
\textrm{the regular $L_\delta$-orbits}\end{array}\right)- \left(\begin{array}{l}\# \, \textrm{$p$-tuples with all entries from} \\
\textrm{the same regular $L_\delta$-orbit}\end{array}\right)\\
   &=(k|L_\delta|)^p-k(|L_{\delta}|)^p\\
   &=|L_{\delta}|^p(k^p-k)
\end{align*}
is even.
\end{proof}

\subsection{Probabilistic methods}\label{ss:prob}

Let $G \leqslant \Sym(\Omega)$ be a transitive permutation group of degree $n$ with point stabiliser $H$. As initially observed by Liebeck and Shalev in \cite{LSh99}, probabilistic methods can be an effective tool for determining upper bounds on the base size $b(G)$. For example, a probabilistic approach plays a central role in the proof of Cameron's conjecture on the base sizes of almost simple primitive groups (see \cite{Bur7,BGS,BLS,BOW}). In order to establish an upper bound $b(G) \leqs c$, the basic idea is to show that the probability 
\begin{equation}\label{e:qgc}
Q(G,c) = \frac{|\{(\a_1, \ldots, \a_c) \in \O^c \,:\, G_{\a_1, \ldots, \a_c} \ne 1\}|}{n^c}
\end{equation}
that a randomly chosen $c$-tuple in $\O$ is \emph{not} a base for $G$ is less than $1$, which guarantees the existence of a base of size at most $c$. In this section we use the same approach to establish lower bounds on the valency of the Saxl graph of a base-two group. 

Suppose $b(G)=2$. Fix a point $\a \in \O$ and observe that 
\[
\frac{{\rm val}(G)}{n} = \frac{|\{\b \in \O \,:\, G_{\a,\b}=1\}|}{n} = \frac{|\{(\a_1,\a_2) \in \O^2 \,:\, G_{\a_1,\a_2}=1\}|}{n^2}
\]
and thus
\[
1-\frac{{\rm val}(G)}{n} = \frac{|\{(\a_1,\a_2) \in \O^2 \,:\, G_{\a_1,\a_2} \ne 1\}|}{n^2} =Q(G,2).
\]
Recall that the \emph{clique number} of a graph $\Gamma$ is the maximal size of a complete subgraph. In addition, we say that $\Gamma$ is \emph{Hamiltonian} if it contains a cycle that visits every vertex exactly once.

\begin{lemma}\label{l:q2}
Let $G$ and $H$ be as above and set  
$$t := \max\{m \in \mathbb{N}\,:\, Q(G,2)< 1/m\}.$$ 
If $t \geqs 2$, then $\Sigma(G)$ has the following properties:
\begin{itemize}\addtolength{\itemsep}{0.2\baselineskip}
\item[{\rm (i)}] Any $t$ vertices in $\Sigma(G)$ have a common neighbour.
\item[{\rm (ii)}] Every edge in $\Sigma(G)$ is contained in a complete subgraph of size $t+1$.
\item[{\rm (iii)}] The clique number of $\Sigma(G)$ is at least $t+1$.
\item[{\rm (iv)}] $\Sigma(G)$ is connected with diameter at most $2$.
\item[{\rm (v)}] $\Sigma(G)$ is Hamiltonian. 
\end{itemize}
\end{lemma}

\begin{proof}
Since $Q(G,2)<1/t$, we have ${\rm val}(G)>n(1-1/t)$. Now given any set of size $n$ and integers $2 \leqs k \leqs t$, any $k$ subsets of size greater than $n(1-1/t)$ have common intersection of size greater than $n(1-k/t)$. Setting $k=t$, we deduce that the common intersection of the neighbourhoods of any $t$ vertices in $\Sigma(G)$ is non-empty. This gives (i), and (ii), (iii) and (iv) follow immediately. Finally, the Hamiltonicity of $\Sigma(G)$ follows from Dirac's theorem (see \cite[Theorem 3]{Dirac}) since ${\rm val}(G)>n/2$.
\end{proof}

\begin{remark}\label{r:ham}
There are only five known finite connected non-Hamiltonian vertex-transitive graphs (with degree $2,10,28,30,84$ and respective valency $1,3,3,3,3$), and a famous conjecture of Lov\'{a}sz \cite[Problem 11]{Lo} asserts that these are the only examples. Of course, if we take the natural action of $G=S_2$, then $\Sigma(G)$ is the complete graph $K_2$, which is trivially non-Hamiltonian. By Proposition \ref{p:prime}, none of the other known non-Hamiltonian graphs can arise as the Saxl graph of a finite permutation group. 
\end{remark}

In view of Lemma \ref{l:q2}, we are interested in determining upper bounds on $Q(G,2)$. For $x \in G$, let ${\rm fix}(x,G/H)$ be the number of fixed points of $x$ on $G/H$ and let 
\[
{\rm fpr}(x,G/H) = \frac{{\rm fix}(x,G/H)}{n} = \frac{|x^G \cap H|}{|x^G|}
\]
be the \emph{fixed point ratio} of $x$. The key result is the following (see the proof of Theorem 1.3 in \cite{LSh99}).

\begin{lemma}\label{l:key}
Let $G \leqslant \Sym(\Omega)$ be a finite transitive group with point stabiliser $H$. Then
\[
Q(G,2) \leqs\sum_{i=1}^{k}|x_i^G|\cdot {\rm fpr}(x_i,G/H)^2 =: \widehat{Q}(G),
\]
where $x_1, \ldots, x_k$ represent the conjugacy classes of elements of prime order in $G$.
\end{lemma}

As an immediate corollary, note that if
\begin{equation}\label{e:star}
|H|^2\max_{1 \ne x \in H}|C_G(x)| < \frac{1}{2}|G|, 
\end{equation}
then $\widehat{Q}(G)<1/2$ and thus $\Sigma(G)$ satisfies all of the properties in Lemma \ref{l:q2} (with $t=2$). In particular, any two vertices in $\Sigma(G)$ have a common neighbour. 

\subsection{Asymptotics}\label{ss:asymp}

It is interesting to consider the asymptotic behaviour of $Q(G,2)$ for infinite families of base-two groups $G$. For example, let $G \leqs {\rm Sym}(\O)$ be an almost simple primitive permutation group with socle $T = A_n$ and point stabiliser $H$. Assume that $H \cap T$ acts primitively on $\{1, \ldots, n\}$. In this situation, Cameron and Kantor showed that $Q(G,2) \to 0$ as $n$ tends to infinity (the sketch proof of \cite[Proposition 2.3]{CK}, which is based on the observations that the order of $H$ is small and the minimal degree of $H$ is large, can be formalised by arguing as in the proof of Lemma \ref{l:n42} in Section \ref{s:sym} below). In other words, for large $n$, almost every pair of points in $\O$ form a base for $G$. Further families of primitive groups with the same asymptotic behaviour have been identified by Fawcett in \cite{Fawcett1, Fawcett2}.

In stark contrast, the next example presents a family of base-two primitive groups with very different asymptotic properties.

\begin{example}\label{ex:pgl2q}
Let $q>5$ be a prime power and consider the action of $G={\rm PGL}_2(q)$ on the set $\O$ of pairs of $1$-dimensional subspaces of $\mathbb{F}_{q}^2$. Fix $\a \in \O$ and recall that $G_{\a} = D_{2(q-1)}$ has a unique regular orbit (see Example \ref{ex:johnson}), so 
$$Q(G,2)=1 - \frac{{\rm val}(G)}{|\O|} = 1-\frac{4(q-1)}{q(q+1)}$$
and thus $Q(G,2) \to 1$ as $q \to \infty$. Even though $Q(G,2)>1/2$ for all $q>5$, it is worth noting that $\Sigma(G)$ still has the property that any two vertices have a common neighbour (this follows immediately from the fact that $\Sigma(G)$ is isomorphic to the Johnson graph $J(q+1,2)$).
\end{example}
 
By considering wreath products in the product action, we can construct further families of primitive groups with the same asymptotic behaviour. To do this, we need a preliminary result, which shows that Lemma \ref{lem:findbase} can be used to build groups with a unique regular suborbit. 

\begin{lemma}\label{lem:unique}
Let $L \leqs {\rm Sym}(\Delta)$ be a transitive base-two group such that 
$L_{\delta}$ has $k$ regular orbits for some $k \geqs 2$. Set $G = L\Wr S_k$ with its product action on $\Omega=\Delta^k$. Then $G$ has a unique regular suborbit.
\end{lemma}

\begin{proof}
First observe that the standard action of $S_k$ on $k$ points has distinguishing number $k$. Let us also note that $L$ has exactly $k$ orbits on the set of ordered bases for $L$ of size $2$, so Corollary \ref{cor:base2} implies that $G = L\Wr S_k$ has base size $2$. 

Fix $\delta\in\Delta$ and set $\a=(\delta,\ldots,\delta)\in\Omega$. By Lemma \ref{lem:findbase}, $\{\a,(\gamma_1,\ldots,\gamma_k)\}$ is a base for $G$ if and only if the $\gamma_i$ lie in distinct regular $L_\delta$-orbits. Thus the degree of $\a$ in $\Sigma(G)$ is
\[
k|L_\delta| \cdot (k-1)|L_\delta| \cdots |L_\delta| = k!|L_\delta|^k
  =|G_\a|.
  \]
Since $G_\a$ acts semiregularly on the set of neighbours of $\a$, it follows that 
$G_\a$ has a unique regular orbit.
\end{proof}

\begin{example}\label{c:pex}
Let $p$ be an odd prime and consider the standard action of $L =D_{2p}$ on a set $\Delta$ of size $p$. Now $G= L \Wr S_{(p-1)/2}$ acts on $\O = \Delta^{(p-1)/2}$ via the product action, which allows us to view  
\[
G = (C_p)^{(p-1)/2}{:}(C_2 \Wr S_{(p-1)/2}) < {\rm AGL}_{(p-1)/2}(p)
\]
as a primitive affine group. By Lemma \ref{lem:unique}, ${\rm val}(G) = 2^{(p-1)/2}( (p-1)/2)!$ and thus 
$$Q(G,2) = 1 - \frac{{\rm val}(G)}{|\O|} \to 1 \mbox{ as $p \to \infty$.}$$
\end{example}

\begin{example}\label{r:pp}
We can also use Lemma \ref{lem:unique} to build similar examples from an almost simple primitive group $L$. For instance,  
let $L = {\rm L}_{2}(p)$, where $p \geqs 5$ is a prime such that $p \equiv \pm 3 \imod{8}$ and $p \not\equiv \pm 1 \imod{10}$. Then $H = A_4$ is a maximal subgroup of $L$ and we can consider the primitive action of $L$ on $\Delta = L/H$. For convenience, let us also  assume that $p \equiv 1 \imod{12}$. Let $k$ be the number of regular $H$-orbits on $\Delta$. 
If $x \in L \setminus H$ then $|H \cap H^x| \in \{1,2,3\}$, so the non-trivial orbits of $H$ have length $12$, $6$ and $4$. Write $H = L_{\a}$ for some $\a \in\Delta$.

Suppose $u \in H$ has order $2$. Then $|u^L \cap H|=3$, $|u^L| = p(p+1)/2$ and thus $u$ has precisely $|u^L \cap H||C_L(u)|/|H| = (p-1)/4$ fixed points on $\Delta$ (one of which is $\a$). Now $u$ has cycle-shape $(2^2,1^2)$ on $H/\la u \ra$, so $H$ has $((p-1)/4-1)/2 = (p-5)/8$ orbits of length $6$. Now assume $v \in H$ has order $3$, so $|v^L \cap H|=8$, $|v^L|=p(p+1)$ and $v$ has a unique fixed point on $H/\la v \ra$. It quickly follows that $H$ has $(p-4)/3$ orbits of length $4$. Putting this together, we conclude that 
$$k = \frac{|L:H| - 6(p-5)/8 - 4(p-4)/3 - 1}{|H|} = \frac{p^3-51p+194}{288}$$
(a very similar formula holds if either $p \equiv 2 \imod{3}$ or $p \equiv 3 \imod{4}$). 

If we now set $G=L\Wr S_k$ with its product action on $\O = \Delta^k$, then Lemma \ref{lem:unique} implies that $G$ has a unique regular orbit and thus ${\rm val}(G)=12^kk!$. Since $|\O| = (p(p^2-1)/24)^k$, we deduce that $Q(G,2)\rightarrow 1$ as $|G|\rightarrow \infty$.
\end{example}

\begin{remark}\label{r:pal}
We can also find families of base-two primitive groups $G$ with the property that $Q(G,2) \to c$ as $|G| \to \infty$, for some $c \in (0,1)$. For instance, if $G = V{:}D_{q+1}$ is the primitive group constructed in Example \ref{ex:payley}, where $V = \mathbb{F}_{q^2}$ and $q$ is an odd prime power, then 
\[
Q(G,2)=\frac{1}{2}\left(\frac{q^2+1}{q^2}\right) \rightarrow \frac{1}{2}
\]
as $q$ tends to infinity. 
\end{remark}

\section{Connectivity}\label{s:con}

As before, let $G \leqslant \Sym(\Omega)$ be a transitive permutation group of degree $n$ with point stabiliser $H$ and Saxl graph $\Sigma(G)$. In this section, we investigate the connectivity  of $\Sigma(G)$. If $\Sigma(G)$ is connected, we will write ${\rm diam}(G)$ for the diameter of $\Sigma(G)$ (if $\Sigma(G)$ is disconnected then it will be convenient to set ${\rm diam}(G)=\infty$). As noted in Lemma \ref{l:1}, if $G$ is non-regular then ${\rm diam}(G)=1$ if and only if $G$ is a Frobenius group. 

Recall that the connected components of $\Sigma(G)$ form a system of imprimitivity for $G$. In particular, $\Sigma(G)$ is connected if $G$ is primitive, so we start by considering the connectivity properties of $\Sigma(G)$ when $G$ is imprimitive. Our first result shows that $\Sigma(G)$ can have an arbitrarily large number of  connected components.

\begin{proposition}\label{p:c2}
Let $p$ be an odd prime number. Then there exists an imprimitive group $G$ such that $b(G)=2$ and $\Sigma(G)$ has exactly $p$ connected components.
\end{proposition}

\begin{proof}
Let $P$ be an extraspecial group of order $p^3$ and exponent $p$. Write $Z(P)=\langle z\rangle$ and choose $x_1,x_2\in P$ such that $P=\langle z,x_1,x_2\rangle$ and 
$[x_1,x_2]=z$. Recall that $\SL_2(p)$ acts on $P$ as a group of automorphisms by acting trivially on $Z(P)$ and linearly on the elementary abelian quotient $P/Z(P)$ (see \cite[(3F), p.163]{Winter}). In particular, the unique involution $\tau\in\SL_2(p)$ acts on $P$ by centralising $z$ and inverting $x_1$ and $x_2$. Moreover, there exists $\phi\in\SL_2(p)$ of order $p$ that acts on $P$ by centralising $\langle z,x_2\rangle$, and mapping $x_1$ to $x_1x_2$.

Let $G=P{:}\langle \phi,\tau\rangle\cong P{:}(C_p\times C_2)$ and $H=\langle x_1\rangle{:}\langle \tau\rangle\cong D_{2p}$. Consider the action of $G$ on $\Omega = G/H$ and note that $|\Omega|=p^3$. Let $\Sigma(G)$ be the corresponding Saxl graph.

Now $P$ has $p$ orbits on $\O$ of size $p^2$, which form a system of imprimitivity for $G$. More precisely, if $\alpha\in\Omega$ is the element corresponding to $H$ and $\Delta$ is the $P$-orbit containing $\alpha$, then 
$$\Delta^{\phi^i} = \{Hz^ax_2^b\phi^i\,:\, a,b\in\{0,\ldots,p-1\}\},\;\; i = 0, 1, \ldots, p-1$$
are the $P$-orbits on $\O$. Note that $N_P(\langle x_1\rangle)=\langle z,x_1\rangle$, so  $\langle x_1\rangle$ has $p$ fixed points on $\Delta$ (and thus $p-1$ orbits of size $p$). 
Moreover, for $\beta=\alpha^{\phi^i}\in\Omega\setminus \Delta$ we have $P_\beta=\langle x_1\rangle^{\phi^i}=\langle x_1x_2^i\rangle$ and thus $\langle x_1\rangle$  acts semiregularly on each $\Delta^{\phi^i}$, $i = 1, \ldots, p-1$.

Since $\tau\in H$ we have $Hz^ax_2^b\phi^i\tau=H(z^\tau)^a(x_2^\tau)^b(\phi^\tau)^i=Hz^ax_2^{-b}\phi^i$, hence $Hz^ax_2^b\phi^i$ is fixed by $\tau$ if and only if $b=0$ (since $p$ is odd). Thus $\tau$ has $p$ fixed points on $\Delta^{\phi^i}$.  If $i=0$ then these fixed points are of the form $Hz^a$ with $a \in \{0, \ldots, p-1\}$, which coincide with the fixed points of $\langle x_1\rangle$. Therefore, $H$ has $p$ fixed points on $\Delta$ and $(p-1)/2$  orbits of size $2p$. In terms of $\Sigma(G)$, it follows that $\alpha$ is adjacent to $p^2-p$ elements of $\Delta$.
On the other hand, $Hz^a\phi^ix_1=Hz^ax_1x_2^{-i}\phi^i$ and so the fixed points of $\tau$ on $\Delta^{\phi^i}$, for $i = 1, \ldots, p-1$, are not fixed by $\langle x_1\rangle$. In particular, this implies that $\tau$ has one fixed point in each $\langle x_1\rangle$-orbit on $\Delta^{\phi^i}$ and so $H$ has $p$ orbits of size $p$ on $\Delta^{\phi^i}$. Therefore, $|G_{\alpha,\beta}|=2$ for all 
$\beta\in \Omega\setminus \Delta$ and thus each edge of $\Sigma(G)$ is contained within one of the $\Delta^{\phi^i}$. In other words, $\Sigma(G)$ is the union of the graphs induced on each $P$-orbit.

Since $G$ acts transitively on the set of $P$-orbits, the induced graphs are isomorphic and so it remains to determine the graph induced on $\Delta$. First observe that 
$P{:}\langle\tau\rangle$ is the setwise stabiliser in $G$ of $\Delta$. As $G$ is finite, $H$ is the stabiliser in $G$ of each of the $p$ vertices $Hz^a$. Moreover, this set of mutually nonadjacent vertices is an orbit of the normal subgroup $N=\langle x_1,z\rangle$ of $P{:}\langle \tau\rangle$. The orbits of $N$ form a system of imprimitivity for $P{:}\langle \tau\rangle$ and since they are transitively permuted by $P{:}\langle \tau\rangle$, each one consists of $p$ mutually nonadjacent vertices. Since the valency is $p^2-p$, it follows that the graph induced on $\Delta$ is the complete multipartite graph with $p$ parts of size $p$.

We conclude that $\Sigma(G)$ is the union of $p$ copies of the complete multipartite graph with $p$ parts of size $p$. In particular, $\Sigma(G)$ has exactly $p$ connected components.
\end{proof}

Let $n$ and $k$ be positive integers. The \emph{wreath graph} $W(n,k)$ is the graph with vertex set $\mathbb{Z}_k\times \mathbb{Z}_n$ such that $\{(a,i),(b,j)\}$ is an edge if and only if $j\equiv i\pm 1 \imod{n}$.  Note that $W(n,k)$ consists of $n$ sets 
$$\Delta_i=\{(a,i) \,:\, a\in\mathbb{Z}_k\}$$ 
of $k$ mutually non-adjacent vertices arranged in an $n$-cycle so that each vertex of $\Delta_i$ is adjacent to all vertices in $\Delta_{i+1} \cup \Delta_{i-1}$. In particular, $\Aut(W(n,k))=S_k \Wr D_{2n}$.
 
\begin{proposition}\label{p:c3}
Let $n \geqs 2$ be an integer and let $p$ be a prime such that $(p,2n+1)=1$. Then there exists a permutation group $G$ such that $\Sigma(G)=W(2n+1,p^n)$.
\end{proposition}

\begin{proof}
Set $V=\mathbb{F}_{p}^{2n+1}$  and let
$$W=\{(v_1,\ldots,v_{2n+1})\in V \, :\, \sum v_i =0\}$$ 
be the deleted permutation module for $S_{2n+1}$. Let $g\in S_{2n+1}$ be the cycle $(1,2,\ldots,2n+1)$ and define the following vectors in $W$:
\begin{align*}
 e_1 &=(p-2n,1,1,\ldots,1)\\
 e_2 &= (1,p-2n,1,1,\ldots,1)\\
 e_3 &=(1,1,p-2n,1,1,\ldots,1)\\
        &\;\; \vdots\\
 e_{2n+1}&=(1,\ldots,1,p-2n)
\end{align*}
Note that the $e_i$ are all distinct (since $p$ does not divide $2n+1$). Moreover, $e_{2n+1}= -\sum_{i=1}^{2n}e_i$ and $\{e_1,e_2,\ldots,e_{2n}\}$ is a basis for $W$. Finally, observe that $g$ cyclically permutes the $2n+1$ vectors $e_1,e_2,\ldots,e_{2n+1}$.

Form the semidirect product $G=W{:}\langle g\rangle$, which has order $p^{2n}(2n+1)$, and consider the action of $G$ on $\O = G/U$, where 
$$U=\langle e_1,e_3,\ldots, e_{2n-5},e_{2n-3},e_{2n}\rangle$$
is a subspace of $W$. Now 
$$U^g=\langle e_2,e_4,\ldots,e_{2n-2},e_{2n+1}\rangle$$
and thus $U^g\cap U=\{0\}$. Therefore, $G$ acts faithfully on $\Omega$ and we see that $b(G)=2$ since $U^g$ is the stabiliser in $G$ of the coset $Ug$. Let $\Sigma(G)$ be the corresponding Saxl graph. For the remainder of the proof, it will be convenient to define $U_i = U^{g^i}$ for $i \in \{0,1, \ldots, 2n\}$.

First we claim that the $U_i$ are distinct subspaces of $W$. To see this, it suffices to show that $U = U_i$ if and only if $i=0$. Suppose $U = U_i$. Then $e_1^{g^i}\in U$ and so either $i$ is even with $i\leqslant 2n-4$, or $i=2n-1$. 
Now
$$U_2 =\langle e_3,e_5,\ldots,e_{2n-1},e_1\rangle,\;\; U_3 =\langle e_4,e_6,\ldots, e_{2n},e_2\rangle$$
and we see that $e_{2n+1}$ is contained in $U_{2i+4}$ for all $i\in \{0,\ldots,n-3 \}$. Since $e_{2n+1}\notin U$ it follows that $U \ne U_i$ when $i$ is even and $2 \leqs i \leqs 2n-4$.  Finally, $e_3^{g^{2n-1}}=e_{2n-2}\notin U$ and thus $U \ne U_{2n-1}$. This justifies the claim.  

Let $\Gamma$ be the graph with vertex set $U^{\langle g\rangle}$ so that two subspaces are adjacent if and only if they intersect trivially. We have already noted that $U^g \cap U = \{0\}$, so $(U,U_1)$ is an arc in $\Gamma$ and by repeatedly applying $g$ we see that $(U,U_1,U_2, \ldots, U_{2n})$ is a cycle of length $2n+1$ in $\Gamma$. In fact, we claim that $\Gamma$ coincides with this cycle. That is, if $|i-j|>1$ then there is no edge between $U_{i}$ and $U_{j}$. To see this, first observe that $e_1\in U_{2}\cap U$ and so $e_{2i-1}\in U_{2i} \cap U$ for all $i<n$. Also $e_{2n-1}\in U_{i}$ for all even $i$ such that $2\leqs i\leqs 2n-2$, so $e_{2n}\in U_{i}\cap U$ for all odd $i$ in the range $3\leqs i\leqs 2n-1$. Hence $U$ is not adjacent to any $U_i$ with $1<i<2n$ and we deduce that $\Gamma$ is a cycle of length $2n+1$. In particular, $\Gamma$ has diameter $n$. 

We can identify $\Omega$ with the set of translates of the subspaces in $U^{\langle g\rangle}$. That is, 
\[
\O = \{U_i+v \,:\, v \in W,\; 0 \leqs i \leqs 2n\}.
\]
Notice that under this identification, each subset $\{U_i+v \,:\, v \in W\}$ forms a block of imprimitivity for $G$. Since the stabiliser in $G$ of each translate $U+v$ is $U$, it follows that the stabiliser in $G$ of $U_i+v$ is $U_i$. Therefore, $U_i+v$ and $U_j+u$ are adjacent in $\Sigma(G)$ if and only if $U_i\cap U_j=\{0\}$, that is, if and only if $U_i$ and $U_j$ are adjacent in $\Gamma$. In particular, if there is an edge in $\Gamma$ connecting $U_i$ and $U_j$, then $\Sigma(G)$ induces a complete bipartite graph between the set of translates of $U_i$ and the set of translates of $U_j$. Therefore, $\Sigma(G)=W(2n+1,p^n)$ and the proof is complete. 
\end{proof}

Since $W(2n+1,p^n)$ has diameter $n$, the following corollary is immediate.

\begin{corollary}
For each positive integer $n$, there are infinitely many imprimitive groups $G$ such that $\Sigma(G)$ is connected and ${\rm diam}(G)=n$.
\end{corollary}

The above results show that the Saxl graph of a base-two imprimitive group can exhibit a wide range of connectivity properties. However, it appears that the situation for primitive groups is rather more restrictive. Indeed, we propose the following striking conjecture.  

\begin{conjecture}\label{c:diam}
Let $G$ be a finite primitive permutation group with $b(G)=2$ and Saxl graph $\Sigma(G)$. Then either $G$ is a Frobenius group and $\Sigma(G)$ is complete, or ${\rm diam}(G)=2$.
\end{conjecture} 

In fact, we conjecture that the Saxl graph of a primitive group has the following even stronger property.

\begin{conjecture}\label{c:diam2}
Let $G$ be a finite primitive permutation group with $b(G)=2$. Then any two vertices in the Saxl graph of $G$ have a common neighbour. 
\end{conjecture}

Recall that the structure of a finite primitive permutation group $G$ is described by the O'Nan-Scott theorem (see \cite[Theorem 4.1A]{DM}) and one of the following holds:
\begin{itemize}\addtolength{\itemsep}{0.2\baselineskip}
\item[(a)] $G$ is almost simple;
\item[(b)] $G$ is of diagonal type;
\item[(c)] $G$ is a twisted wreath product;
\item[(d)] $G$ is of affine type;
\item[(e)] $G$ is of product type.
\end{itemize}

In Sections \ref{s:sym} and \ref{s:spor}, we will establish Conjecture \ref{c:diam2} for families of almost simple primitive groups for which we have a complete classification of the base-two examples (namely, when the socle is an alternating or sporadic group). 

For diagonal-type groups and twisted wreath products, evidence for the veracity of the conjectures can be deduced from work of Fawcett \cite{Fawcett1, Fawcett2}. More precisely, suppose $G$ is a primitive group of diagonal type with socle $T^k$, where $T$ is a non-abelian simple group and $k \geqs 2$. Let $P_G \leqs S_k$ be the group of permutations of the $k$ factors of $T^k$ induced by $G$. By \cite[Theorem 1.1]{Fawcett2}, if $P_G \ne A_k,S_k$, then $b(G)=2$. Moreover, in this situation \cite[Theorem 1.4]{Fawcett2} states that $Q(G,2) \to 0$ as $|G|\to \infty$ (where $Q(G,2)$ is defined in \eqref{e:qgc}, with $c=2$), so by applying Lemma \ref{l:q2} we get the following asymptotic version of Conjecture \ref{c:diam2} for diagonal type groups.

\begin{proposition}\label{p:diag1}
Any two vertices in $\Sigma(G)$ have a common neighbour for all sufficiently large primitive groups $G$ of diagonal type with socle $T^k$ and $P_G \ne A_k, S_k$.
\end{proposition}

In addition, for each fixed $k \geqs 5$, \cite[Theorem 1.5]{Fawcett2} implies that the same conclusion holds for any diagonal type group $G$ (without any condition on $P_G$). Similarly, if $G=T^k{:}P$ is a twisted wreath product and $P \leqs S_k$ is primitive, then \cite[Theorem 5.0.1]{Fawcett1} states that $b(G)=2$ and by applying \cite[Theorem 5.0.2]{Fawcett1} we get the following result.

\begin{proposition}\label{p:tw1}
Any two vertices in $\Sigma(G)$ have a common neighbour for all sufficiently large primitive twisted wreath products $G = T^k{:}P$ such that $P \leqs S_k$ is primitive. 
\end{proposition}

There are very few results in the literature on base-two primitive groups in the remaining cases (d) and (e) above, but there is some evidence to suggest that our conjectures are still valid in these cases. For example, we have verified Conjecture \ref{c:diam2} for all primitive groups in the {\sc Magma} database (that is, all primitive groups of degree at most $4095$). Note that this includes all primitive subgroups of ${\rm AGL}_{d}(q)$ where $q = 2,3,5$ and $d \leqs 11,7,5$, respectively. It is also worth noting that the conjecture holds for the infinite family of primitive affine groups $G$ in Example \ref{ex:payley}, where $\Sigma(G)$ is a Paley graph.

\begin{remark}
We have also verified Conjecture \ref{c:diam2} for an infinite family of primitive groups $G$ for which $Q(G,2) \to 1$ as $|G|$ tends to infinity. Indeed, if $q > 5$ then the action of $G = {\rm PGL}_{2}(q)$ on the set $\O$ of cosets of $D_{2(q-1)}$ is primitive and we showed in Example \ref{ex:pgl2q} that $Q(G,2) \to 1$ as $q \to \infty$. Moreover, even though the probability that a randomly chosen pair of points in $\O$ forms a base for $G$ tends to $0$, we still observe that any two vertices in $\Sigma(G)$ have a common neighbour. It would be interesting to show that this property also holds for the groups in Examples \ref{c:pex} and \ref{r:pp}. For instance, with the aid of {\sc Magma}, we have checked that the affine groups $G < {\rm AGL}_{(p-1)/2}(p)$ in Example \ref{c:pex} satisfy the conjecture when $p \leqs 13$.
\end{remark}

Let $\omega(G)$ be the \emph{clique number} of $\Sigma(G)$, which is the maximal size of a complete subgraph of $\Sigma(G)$. That is, $\omega(G)$ is the size of the largest subset of $\Omega$ such that any pair of elements forms a base. In this terminology, Conjecture \ref{c:diam2} asserts that if $G$ is primitive then every edge in $\Sigma(G)$ is contained in a triangle and thus $\omega(G) \geqs 3$. Note that this property does not extend to all base-two transitive groups. For example, it is easy to see that $\omega(G)=2$ for the imprimitive groups $G$ we constructed in the proof of Proposition \ref{p:c3}. Also observe that we can construct primitive groups $G$ with the property that $\omega(G)$ is arbitrarily large. For instance, $\omega(G)=q$ when $G = \mathbb{F}_{q^2}{:}D_{q+1}$ is a primitive affine group as in Example \ref{ex:payley} (see \cite{BDR}). In view of Lemma \ref{l:q2}(iii), further examples can be constructed from any family of base-two groups $G$ such that $Q(G,2) \to 0$ as $|G|$ tends to infinity (see Section \ref{ss:asymp}).

\section{Symmetric and alternating groups}\label{s:sym}

Let $G \leqs {\rm Sym}(\O)$ be an almost simple primitive group with socle $T=A_n$ and point stabiliser $H$. Suppose $G = S_n$ or $A_n$ and consider the action of $H$ on 
$\{1, \ldots, n\}$. If $H$ acts primitively, then the groups with $b(G)=2$ are determined in \cite{BGS1}. Similarly, if $H$ is transitive but imprimitive, then James \cite{James2006, James_thesis} has classified the groups with $b(G)=2$ (see \cite[Theorem 1.2 and Remark 5.3]{James2006}). Finally, if $H$ is intransitive then $H = (S_k  \times S_{n-k}) \cap G$ with  $1 \leqs k < n/2$ and it is easy to check that the trivial bound
$$b(G) \geqs \frac{\log |G|}{\log |\O|}$$
gives $b(G)>2$, unless $G = A_5$ and $k=2$, in which case $b(G)=2$. See \cite[Remark 1.7]{BGS1} for the handful of additional groups that arise when $n=6$ and $G \ne S_6,A_6$. 

In view of the above results, we have a complete classification of the groups of this form with $b(G)=2$. For instance, if $n>12$ then $b(G)=2$ if and only if one of the following holds:
\begin{itemize}\addtolength{\itemsep}{0.2\baselineskip}
\item[{\rm (a)}] $H$ acts primitively on $\{1, \ldots, n\}$;
\item[{\rm (b)}] $H = (S_k \Wr S_\ell) \cap G$, where $n=k\ell$, $k \geqs 3$ and either $G=S_n$ and $\ell \geqs \max\{8,k+3\}$, or $G=A_n$ and $\ell \geqs \max\{7-\delta_{3,k},k+2\}$.
\end{itemize}
We refer the reader to \cite{James2006} for the conditions on $k$ and $\ell$ in (b).

In this section, we establish Conjecture \ref{c:diam2} in case (a). In the statement of the next result, we define $Q(G,2)$ as in \eqref{e:qgc} (with $c=2$).

\begin{theorem}\label{t:main1}
Let $G$ be an almost simple primitive permutation group with socle $T = A_n$ and point stabiliser $H$. Assume that $b(G)=2$ and $H \cap T$ acts primitively on $\{1,\ldots, n\}$. Then one of the following holds:
\begin{itemize}\addtolength{\itemsep}{0.2\baselineskip}
\item[{\rm (i)}] $Q(G,2)<1/2$;
\item[{\rm (ii)}] $G=S_7$, $H = {\rm AGL}_{1}(7)$ and $Q(G,2) = 13/20$.
\end{itemize}
Moreover, any two vertices in the Saxl graph $\Sigma(G)$ have a common neighbour, so $\Sigma(G)$ has diameter $2$ and every edge in $\Sigma(G)$ is contained in a triangle.
\end{theorem}

In view of Lemma \ref{l:q2}, it suffices to show that $Q(G,2)<1/2$ (the special case $(G,H) = (S_7, {\rm AGL}_{1}(7))$ will be handled separately). We partition the proof into two cases.

\begin{lemma}\label{l:n42}
Theorem \ref{t:main1} holds if $n \geqs 42$. 
\end{lemma}

\begin{proof}
First recall that $Q(G,2) \leqs \widehat{Q}(G)$ (see Lemma \ref{l:key}), so it suffices to show that $\widehat{Q}(G)<1/2$. We proceed as in the proof of \cite[Theorem 1.1]{BGS1}. 

To begin with, let us exclude the following three cases:
\begin{itemize}\addtolength{\itemsep}{0.2\baselineskip}
\item[(i)] $H = (S_\ell \Wr S_k) \cap G$ in the product action on $n=\ell^k$ points;
\item[(ii)] $H=A_\ell$ or $S_\ell$ acting on $m$-element subsets of $\{1, \ldots, \ell\}$ (so $n = \binom{\ell}{m}$);
\item[(iii)] $H$ is an almost simple orthogonal group over $\mathbb{F}_2$ acting on a set of hyperplanes of the natural module for $H$.
\end{itemize}
Firstly, a theorem of Mar\'{o}ti \cite[Theorem 1.1]{Maroti} gives
\[
|H| \leqs n^{1+\lfloor \log_2n \rfloor}.
\]
In addition, a result of Guralnick and Magaard \cite[Theorem 1]{GM} states that the minimal degree of $H$ is at least $n/2$, which implies that 
\[
|C_G(x)| \leqs 2^{n/4}\lceil n/4\rceil!\lceil n/2\rceil!
\]
for all $1 \ne x \in H$. Since $n \geqs 42$, it is now straightforward to verify that the bound in \eqref{e:star} is satisfied and thus the desired result follows immediately from Lemma \ref{l:q2}.

It remains to consider the cases (i), (ii) and (iii). Suppose (i) holds. We will assume $G=S_n$ and $k=2$, so $H = S_\ell \Wr S_2$, $n=\ell^2$ and $\ell \geqs 7$ (the other cases are similar and easier). For $i=1,2$ we define 
\[
Q_i = \sum_{x \in X_i}|C_G(x)||x^G \cap H|^2,
\]
where $X_1$ (respectively, $X_2$) is a set of representatives of the conjugacy classes of elements in $G$ of order $2$ (respectively, odd prime order). Note that 
$|G|\cdot \widehat{Q}(G) = Q_1+Q_2$. Let $i_2(H)$ be the number of involutions in $H$ and let $p(\ell)$ be the number of partitions of $\ell$. As explained in the proof of  \cite[Theorem 1.1]{BGS1}, we have
\[
Q_1 < i_2(H)^22^\ell\ell!(\ell^2-2\ell)!< (\ell!)^3(1+p(\ell))^42^\ell(\ell^2-2\ell)! < (\ell-1)^{4\ell}(\ell^2-2\ell)!
\]
and
\[
Q_2 < |H|^23^\ell\ell!(\ell^2-3\ell)! = 4(\ell!)^53^\ell(\ell^2-3\ell)!
\]
and it is easy to check that $Q_i<|G|/4$ for $i=1,2$. Therefore, $\widehat{Q}(G)<1/2$ and the result follows via Lemmas \ref{l:q2} and \ref{l:key}.

Next consider (ii), so $H = A_\ell$ or $S_\ell$ and $n = \binom{\ell}{m}$. We will assume $G = S_n$ and $m=2$, so $\ell \geqs 10$ (again, the other cases are similar and easier). If $x \in H$ has prime order, then one checks that $|C_G(x)| \leqs 2^{\ell-2}(\ell-2)!\lceil (\ell^2-5\ell+8)/2\rceil!$ and thus
\[
|H|^22^{\ell-2}(\ell-2)!\lceil (\ell^2-5\ell+8)/2\rceil! \leqs 2^{\ell-2}(\ell!)^2(\ell-2)!\lceil (\ell^2-5\ell+8)/2\rceil! < \frac{1}{2}|G|.
\]
Therefore \eqref{e:star} holds and we conclude that $\what{Q}(G)<1/2$.

Finally, let us turn to (iii). Let $\mu(H)$ denote the minimal degree of $H$ as a permutation group of degree $n$. As noted in the proof of \cite[Theorem 1.1]{BGS1}, if $H$ is of type $O_{2\ell+1}(2)$ or $O_{2\ell}^{+}(2)$, then $n = 2^{\ell-1}(2^\ell-1)$ (so $\ell \geqs 4$) and $\mu(H) = n/2-2^{\ell-2}=\a$. Therefore, $|H|<2^{2\ell^2+\ell}$ and $|C_G(x)| \leqs 2^{\a/2}\lceil \a/2\rceil!\lceil n-\a \rceil!$ for all $1 \ne x \in H$. In particular, it is straightforward to check that \eqref{e:star} holds. Similarly, if $H$ is of type $O_{2\ell}^{-}(2)$ then $n=(2^\ell+1)(2^{\ell-1}-1)$ and $\mu(H) = n/2-(2^{\ell-1}-1)/2=\b$, which implies that $|H|< 2^{2\ell^2-\ell+1}$ and $|C_G(x)| \leqs 2^{\b/2}\lceil \b/2\rceil!\lceil n-\b \rceil!$ for all $1 \ne x \in H$. Again, the desired result follows via \eqref{e:star}.
\end{proof}

\begin{lemma}\label{l:n421}
Theorem \ref{t:main1} holds if $n<42$.
\end{lemma}

\begin{proof}
First assume $G=S_n$. For $13 \leqs n \leqs 41$, it is easy to check that \eqref{e:star} holds with the aid of {\sc Magma} \cite{magma}, unless $n \in \{14,18, 21,22\}$. Here the relevant cases $(G,H)$ are as follows:
\[
(S_{22}, {\rm M}_{22}.2),\; (S_{21}, {\rm L}_{3}(4).S_3), \; (S_{18}, {\rm PGL}_{2}(17)), \;  (S_{14}, {\rm PGL}_{2}(13)).
\]
In each of these cases, it is straightforward to check that $\widehat{Q}(G)<1/2$ as required. Finally, let us consider the relevant groups $G=S_n$ with $b(G)=2$ and $5 \leqs n \leqs 12$. Here $(G,H)$ is one of the following:
\[
(S_{12}, {\rm PGL}_{2}(11)), \; (S_{11}, {\rm AGL}_{1}(11)), \; (S_{7}, {\rm AGL}_{1}(7)).
\]
The first two cases are handled as above. However, in the latter case we find that $\widehat{Q}(G)=73/60$, so this needs further attention. Here $G$ has degree $120$ and one can check that 
$H$ has a unique regular orbit, so $\Sigma(G)$ has valency $42$ and we get $Q(G,2) = 13/20$. To handle this case, we can use {\sc Magma} to show that for a fixed $\a \in \O$, if $\b \in \O$ is any other point then there exists $\gamma \in \O$ such that $\{\a,\gamma\}$ and $\{\b,\gamma\}$ are bases. It follows that any two vertices in $\Sigma(G)$ have a common neighbour.  

Similar reasoning applies when $G=A_n$. If $13 \leqs n \leqs 41$ then \eqref{e:star} holds unless $(G,H)$ is one of the following:
\[
(A_{32}, {\rm AGL}_{5}(2)),\; (A_{28}, {\rm Sp}_{6}(2)),\; (A_{24}, {\rm M}_{24}),\; (A_{23}, {\rm M}_{23}),\; (A_{22}, {\rm M}_{22}),
\]
\[(A_{17}, {\rm L}_{2}(16).4),\; (A_{16}, {\rm AGL}_{4}(2)),\; (A_{15}, A_8), \; (A_{14}, {\rm L}_{2}(13)), \; (A_{13}, {\rm L}_{3}(3)).
\] 
In each of these cases, it is easy to check that $\widehat{Q}(G)<1/2$. Finally, let us assume $5 \leqs n \leqs 12$. Here the only options with $b(G)=2$ are $(G,H) = (A_{10},{\rm M}_{10})$ and $(A_{9}, 3^2{:}2A_4)$. In both cases, we have $\widehat{Q}(G)>1/2$, but one can check that $H$ has two regular orbits and this gives $Q(G,2)<1/2$. 

Finally, suppose $n=6$ and $G$ is one of ${\rm Aut}(A_6)$, ${\rm PGL}_{2}(9)$ or ${\rm M}_{10}$. It is straightforward to check that there are no cases such that $b(G)=2$ and $H \cap T$ acts primitively on $\{1, \ldots, 6\}$.
\end{proof}

\begin{corollary}\label{c:ham}
Let $G$ be an almost simple primitive permutation group with socle $T = A_n$, point stabiliser $H$ and base size $2$. Assume $H \cap T$ acts primitively on $\{1,\ldots, n\}$. Then the Saxl graph of $G$ is Hamiltonian.
\end{corollary}

\begin{proof}
By Theorem \ref{t:main1}, either $Q(G,2)<1/2$ or $(G,H) = (S_{7}, {\rm AGL}_{1}(7))$. In the first case, the Hamiltonicity of $\Sigma(G)$ follows from Lemma \ref{l:q2}. In the latter, one can use {\sc Magma} to find a Hamiltonian cycle in $\Sigma(G)$ by random search.
\end{proof}

\begin{remark}\label{r:imp}
Let $G \leqs {\rm Sym}(\O)$ be an almost simple primitive group with socle $T = A_n$ and point stabiliser $H$. As noted above, there are examples where $b(G)=2$ and $H \cap T$ acts imprimitively on $\{1, \ldots, n\}$. Here $\O$ is the set of partitions of $\{1, \ldots, n\}$ into $\ell$ parts of size $k$, and the precise conditions on $k$ and $\ell$ were determined by James \cite{James2006} using a constructive approach (see case (b) in the discussion preceding the statement of Theorem \ref{t:main1}). It seems rather difficult to compute the valency of $\Sigma(G)$ in these cases, both theoretically and computationally. 

For example, it is easy to see that $\widehat{Q}(G)>1$, so our probabilistic approach is not effective. Indeed, if $G = S_n$, $H = S_k \Wr S_\ell$ and $x \in G$ is a transposition, where $k$ and $\ell$ satisfy the bounds in (b) above, then 
$$\widehat{Q}(G)>|x^G|\cdot {\rm fpr}(x,G/H)^2 = \frac{1}{2}n(n-1) \cdot \left(\frac{k-1}{n-1}\right)^2 = \frac{1}{2}\left(\frac{n}{n-1}\right)(k-1)^2>1.$$

In addition, computational methods are difficult to apply due to the size of $\O$. The  smallest degree arises when $G = A_{18}$ and $H = (S_3 \Wr S_6) \cap G$, in which case $|\O| = 190590400$. Using {\sc Magma}, we can find a complete set $R$ of $(H,H)$ double coset representatives (there are $103$ double cosets in total) and we deduce that $H$ has a unique regular orbit on $\O$ (we thank Eamonn O'Brien for his assistance with this double coset computation). Therefore  
\[
Q(G,2) = \frac{543107}{595595}>\frac{1}{2}.
\] 
In order to verify Conjecture \ref{c:diam2} in this case, it suffices to show that for each $x \in R$, there exists $y \in G$ such that 
\[
H \cap H^y = H^x \cap H^y = 1,
\]
and this is easily checked by random search. The next smallest degree is $|\O| = 36212176000$, which corresponds to the case where $G = A_{21}$ and $H = (S_3 \Wr S_7) \cap G$. Here we have been unable to compute a complete set of $(H,H)$ double coset representatives and so we have not verified Conjecture \ref{c:diam2} in this case. 

Given the significant limitations of the probabilistic and computational methods, it seems reasonable to expect that a constructive approach might be the best way to investigate the Saxl graphs of these groups.
\end{remark}

\section{Sporadic groups}\label{s:spor}

In this section we investigate Conjecture \ref{c:diam2} for almost simple sporadic groups. 
Let $G \leqs {\rm Sym}(\O)$ be an almost simple primitive group with socle $T$ and point stabiliser $H$, where $T$ is a sporadic simple group. The exact base size of $G$ has been computed in \cite{BOW} (also see \cite{NNOW}) and thus the groups with $b(G)=2$ are known.

In order to state our first result, define the following collection of sporadic groups:
\[
\mathcal{A} = \{{\rm M}_{11},{\rm M}_{12},{\rm M}_{22},{\rm M}_{23},{\rm M}_{24}, {\rm J}_{1}, {\rm J}_{2}, {\rm J}_{3}, {\rm HS}, {\rm Suz}, {\rm McL}, {\rm Ru}, {\rm He}, {\rm O'N}, {\rm Co}_{2}, {\rm Co}_{3},   {\rm Fi}_{22} \}.
\]

\begin{theorem}\label{t:main2}
Let $G \leqs {\rm Sym}(\O)$ be an almost simple primitive group with socle $T$ and point stabiliser $H$, where $T \in \mathcal{A}$ is a sporadic simple group. If $b(G)=2$ then any two vertices in the Saxl graph of $G$ have a common neighbour.  
\end{theorem}

\begin{proof}
Set $n=|\O|$ and assume $b(G)=2$. As before, let ${\rm val}(G)$ be the valency of $\Sigma(G)$ and recall Lemma \ref{l:1}(v), which states that ${\rm val}(G) = r|H|$, where $r \geqs 1$ is the number of regular orbits of $H$ on $\O$. To prove the theorem, we will adopt a computational approach, using both \textsf{GAP} \cite{GAP} and {\sc Magma} \cite{magma}. This will rely on the probabilistic methods discussed in Section \ref{ss:prob} and we freely use the notation introduced in that section. 

First we use \textsf{GAP} to compute the exact value of $\widehat{Q}(G)$, exploiting the fact that the Character Table Library \cite{CTblLib} contains the character tables of $G$ and $H$, and also the fusion of $H$-classes in $G$, which allows us to compute precise fixed point ratios. As before, if $\widehat{Q}(G)<1/2$ then $Q(G,2)<1/2$ and the result follows from Lemma \ref{l:q2}. The remaining cases that need to be considered are listed in Table \ref{tab:spor0}, where we record $\widehat{Q}(G)$ to three decimal places.

To handle these cases, we first construct $G$ and $H$ in {\sc Magma}, using a suitable permutation representation for $G$ provided in the Web Atlas \cite{WebAt}. Next we use the command \textsf{CosetAction} to first compute the appropriate action of $G$ on the cosets of $H$, and then to determine the orbits of $H$. In particular, this allows us to compute $r$ (the number of regular orbits of $H$) and thus ${\rm val}(G) = r|H|$. If ${\rm val}(G)>n/2$ then $Q(G,2)<1/2$ and we apply Lemma \ref{l:q2}. In each case, $r$ is recorded in the sixth column of Table \ref{tab:spor0}, and in the final column we write \texttt{v} if ${\rm val}(G)>n/2$. 

Finally, suppose ${\rm val}(G) \leqs n/2$. In these cases, we use {\sc Magma} to show directly that any two vertices in $\Sigma(G)$ have a common neighbour. To do this, we fix a point $\a \in \O$ and we calculate a set of representatives $\{\a_1, \ldots, \a_k\}$ for the non-trivial orbits of $H=G_{\a}$. Then for each $i \in \{1, \ldots, k\}$, we use random search to find a 
point $\gamma \in \O$ such that $\{\a,\gamma\}$ and $\{\a_i,\gamma\}$ are bases. The result follows. 
\end{proof}

{\small
\begin{table}
\[
\begin{array}{lllllll} \hline
G & H & n & \widehat{Q}(G) & {\rm val}(G) & r & \\ \hline
{\rm M}_{11} & 2.S_4 & 165 & 1.169 & 48 & 1 &  \\
{\rm M}_{12} & A_4 \times S_3 & 1320 & 0.540 & 936 & 13 & \texttt{v} \\
{\rm M}_{12}.2 & S_4 \times S_3 & 1320 & 1.268 & 576 & 4 &  \\
 & S_5 & 1584 & 0.906 & 720 & 6 &  \\
{\rm J}_{1} & 2^3.7.3 & 1045 & 0.661 & 840 & 5 & \texttt{v} \\
& 2 \times A_5 & 1463 & 0.774 & 840 & 7 & \texttt{v} \\
& 19{:}6 & 1540 & 0.505 & 1026 & 9 & \texttt{v} \\
{\rm J}_{2} & {\rm L}_{3}(2){:}2 & 1800 & 1.924 & 336 & 1 &  \\
& 5^2{:}D_{12} & 2016 & 1.539 & 600 & 2 &  \\
{\rm J}_{2}.2 & 5^2{:}(4 \times S_3) & 2016 & 1.539 & 600 & 1 &  \\
{\rm J}_{3} & {\rm L}_2(19) & 14688 & 1.980 & 3420 & 1 &  \\
& 2^4{:}(3 \times A_5) & 17442 & 1.546 & 8640 & 3 &  \\ 
& {\rm L}_2(17) & 20520 & 1.266 & 9792 & 4 &  \\
& (3 \times A_6){:}2 & 23256 & 1.496 & 4320 & 2 &  \\ 
& 2^{1+4}{:}A_5 & 26163 & 0.837 & 13440 & 7 & \texttt{v} \\
{\rm J}_{3}.2 & 2^4{:}(3 \times A_5).2 & 17442 & 2.248 & 5760 & 1 &  \\
& {\rm L}_2(17) \times 2 & 20520 & 2.422 & 4896 & 1 &  \\ 
& 3^{2+1+2}{:}8.2 & 25840 & 0.826 & 11664 & 3 &  \\
& 2^{1+4}.S_5 & 26163 & 1.539 & 3840 & 1 &  \\
& 2^{2+4}{:}(S_3 \times S_3) & 43605 & 1.078 & 23040 & 10 & \texttt{v} \\ 
{\rm HS} & 2 \times A_6.2^2 &  15400 & 2.588 & 5760 & 2 &   \\
& 5{:}4 \times A_5 & 36960  &  0.607 & 22800 & 19 & \texttt{v} \\
{\rm HS}.2 & 5^{1+2}.[2^5] & 22176 & 0.718 & 12000 & 3 & \texttt{v} \\ 
& 5{:}4 \times S_5 & 36960 & 0.806 & 19200 & 8 & \texttt{v} \\ 
{\rm Suz} & {\rm M}_{12}{:}2 & 2358720 & 2.680 & 380160 & 2 &   \\ 
& 3^{2+4}{:}2.(A_4 \times 2^2).2 & 3203200 & 0.535 & 2239488 & 16 & \texttt{v} \\ 
{\rm Suz}.2 & 3^{2+4}{:}2.(S_4 \times D_8) & 3203200 & 1.931 & 1119744 & 4 &  \\
& ({\rm PGL}_{2}(9) \times A_5).2  & 10378368 & 0.658 & 6048000 & 70 & \texttt{v} \\
{\rm McL} & {\rm M}_{11} & 113400  & 1.514  & 7920 & 1 &  \\ 
{\rm McL}.2 & 2^{2+4}{:}(S_3 \times S_3) & 779625 & 0.540 & 525312 & 228 & \texttt{v} \\ 
{\rm He} & 2^{1+6}.{\rm L}_{3}(2) & 187425 & 1.926  & 107520 & 5 & \texttt{v} \\ 
& 7^2{:}2.{\rm L}_{2}(7) & 244800 & 1.164 & 164640 & 10 &  \texttt{v} \\ 
& 3.S_7 & 266560 & 1.982 & 75600 & 5 &  \\ 
{\rm He}.2 & 2^{1+6}.{\rm L}_{3}(2).2 & 187425 & 2.678 & 43008 & 1 &  \\ 
& 7^2{:}2.{\rm L}_2(7).2 & 244800 & 1.741 & 98784 & 3 &  \\ 
 & 3.S_7 \times 2 & 266560 & 2.861 & 30240 & 1 &  \\ 
& (S_5 \times S_5).2 & 279888 & 3.547 & 86400 & 3 &  \\
& 2^{4+4}.(S_3 \times S_3).2 & 437325 & 2.872 & 92160 & 5 &  \\
& S_4 \times {\rm L}_3(2).2 & 999600 & 0.649 & 637056 & 79 & \texttt{v} \\ 
{\rm O'N} & {\rm J}_1 & 2624832 & 1.126 & 1228920 & 7 &  \\
& 4.{\rm L}_3(4).2 &  2857239 & 1.232 & 967680 & 6 &  \\
{\rm O'N}.2 & {\rm J}_1 \times 2 &  2624832 & 1.943 & 351120 & 1 &  \\
{\rm Co}_{2} & 3^{1+4}{:}2^{1+4}.S_5 & 45337600 & 0.599  & 29859840 & 32 & \texttt{v} \\ 
{\rm Co}_{3} & {\rm L}_3(4).D_{12} & 2049300 & 3.878  & 241920 & 1 &  \\ 
& 2 \times {\rm M}_{12} & 2608200 & 2.378 & 380160 & 2 &  \\ 
{\rm Fi}_{22} & S_{10} & 17791488 & 4.684 & 3628800 & 1 &  \\ \hline
\end{array}
\] 
\caption{The base-two groups with ${\rm soc}(G) \in \mathcal{A}$ and $\widehat{Q}(G) \geqs 1/2$}
\label{tab:spor0}
\end{table}}

For the remaining almost simple sporadic groups, we have partial results. Set
\[
\mathcal{B} = \{{\rm J}_4, {\rm Ly}, {\rm Co}_{1}, {\rm Fi}_{23}, {\rm Fi}_{24}', {\rm HN}, {\rm Th}, \mathbb{B}\}.
\]

\begin{proposition}\label{p:spor2}
Let $G \leqs {\rm Sym}(\O)$ be an almost simple primitive group with socle $T$ and point stabiliser $H$, where $T \in \mathcal{B}$ is a sporadic simple group. Assume $b(G)=2$. Then one of the following holds:
\begin{itemize}\addtolength{\itemsep}{0.2\baselineskip}
\item[(i)] Any two vertices in the Saxl graph of $G$ have a common neighbour. 
\item[(ii)] $(G,H)$ is one of the cases in Table \ref{tab:spor}.
\end{itemize}
\end{proposition}

\begin{proof}
As in the proof of Theorem \ref{t:main2}, we first observe that the character tables of $G$ and $H$ are available in the \textsf{GAP} Character Table Library, together with the fusion map on classes, which allows us to compute $\widehat{Q}(G)$ precisely. If $\widehat{Q}(G)<1/2$ then the conclusion in (i) holds via Lemma \ref{l:q2}; the remaining cases are the ones listed in Table \ref{tab:spor}.
\end{proof}

{\small
\begin{table}
\[
\begin{array}{llll} \hline
G & H & n & \widehat{Q}(G) \\ \hline 
{\rm J}_4 & 2^{3+12}.(S_5 \times {\rm L}_{3}(2)) & 131358148251 &  0.878  \\
{\rm Ly} & 5^3.{\rm L}_{3}(5) & 1113229656 & 0.839   \\
& 2.A_{11} & 1296826875 & 1.059  \\
{\rm Co}_{1} & 3^2.{\rm U}_4(3).D_8 & 17 681 664 000 & 1.439  \\
& 3^6{:}2.{\rm M}_{12} & 30 005 248 000 & 0.679  \\ 
& (A_5 \times {\rm J}_2){:}2 & 57 288 591 360 & 0.576  \\
{\rm Fi}_{23} & [3^{10}].({\rm L}_{3}(3) \times 2) & 6 165 913 600  & 3.923  \\ 
& S_{12} & 8 537 488 128 & 4.918 \\ 
& (2^2 \times 2^{1+8}).(3 \times {\rm U}_4(2)).2 & 12 839 581 755 & 4.513 \\ 
& 2^{6+8}{:}(A_7 \times S_3) & 16 508 033 685 & 2.293 \\ 
{\rm Fi}_{24}' & 2^2.{\rm U}_6(2).3.2 & 5686767482760 & 4.892 \\
& 2^{1+12}.3.{\rm U}_4(3).2 & 7819305288795 & 2.492  \\
{\rm Fi}_{24} & 2^{1+12}.3.{\rm U}_4(3).2^2 & 7819305288795 & 4.512 \\
& 3^{2+4+8}.(S_5 \times 2S_4) & 91122337546240 & 0.545  \\
& S_4 \times \O_8^{+}(2){:}S_3 & 100087107696576 & 0.821 \\
{\rm HN} & 2^{1+8}.(A_5 \times A_5).2 & 74 064 375 & 1.876  \\
& 5^{1+4}{:}2^{1+4}.5.4 & 136 515 456 & 0.606  \\ 
& 2^6.{\rm U}_4(2) & 164 587 500 & 2.156  \\
& (A_6 \times A_6).D_8 & 263 340 000 & 1.142  \\ 
{\rm HN}.2 & 2^{1+8}.(A_5 \times A_5).2^2 & 74064375 &  2.726  \\ 
& 5{:}4 \times {\rm U}_3(5){:}2 & 108345600 & 0.997  \\
& 5^{1+4}.2^{1+4}.5.4.2 & 136515456 & 0.689  \\
& 2^6.{\rm U}_4(2).2 & 164587500 & 2.554  \\
& (S_6 \times S_6).2^2 & 263340000 & 1.445  \\
& 2^{3+2+6}.(S_3 \times {\rm L}_3(2)) & 264515625 & 0.710  \\
{\rm Th} & 2^{1+8}.A_9 &  976841775 & 1.524   \\
& {\rm U}_3(8).6 & 2742012000 & 0.747  \\
& (3 \times G_2(3)){:}2 & 3562272000  & 0.524 \\
\mathbb{B} & [2^{30}].{\rm L}_{5}(2) & 386968944618506250 & 5.311 \\
& [2^{35}].(S_5 \times {\rm L}_{3}(2)) & 5998018641586846875 & 0.706 \\
& S_3 \times {\rm Fi}_{22}.2 & 5362800438804480000 & 0.669 \\ \hline
\end{array}
\] 
\caption{The base-two groups with ${\rm soc}(G) \in \mathcal{B}$ and $\widehat{Q}(G) \geqs 1/2$}
\label{tab:spor}
\end{table}}

\begin{remark}\label{r:spor}
The computational methods used in the proof of Theorem \ref{t:main2} are difficult to apply when $T \in \mathcal{B}$ because the degree $n$ of $G$ is large. In particular, the {\sc Magma} command \textsf{CosetAction} is not an effective tool for determining the orbits of $H$, so it is difficult to compute $r$ and we cannot find a complete set of orbit representatives, which would allow us to directly verify that any two vertices in $\Sigma(G)$ have a common neighbour (equivalently, we need a complete set of $(H,H)$-double coset representatives). Indeed, it is particularly difficult to find representatives of the smallest $H$-orbits, although the \textsf{GAP} package \textsf{orb} (see \cite{orbo,orb}) might be a useful tool in some cases. We do not pursue this here.
\end{remark}

The collections $\mathcal{A}$ and $\mathcal{B}$ cover $25$ of the $26$ sporadic simple groups; the remaining group is of course the Monster $G = \mathbb{M}$, which presents its own unique challenges. To date $44$ conjugacy classes of maximal subgroups $H$ of $G$ have been determined and it is known that any additional maximal subgroup is almost simple with socle ${\rm L}_2(13)$ or  ${\rm L}_2(16)$ (see Wilson's recent survey article \cite{Wil}). The character table of $G$ is available in the \textsf{GAP} Character Table Library, together with the character table of many, but not all, of its maximal subgroups, as well as class fusion information that allows one to compute $\widehat{Q}(G)$ (in some cases, we work with the \textsf{GAP} command \textsf{PossibleClassFusions}). 
In this way, we can establish the following partial result.

\vspace{1mm}

\begin{proposition}\label{p:mon}
Let $G \leqs {\rm Sym}(\O)$ be a primitive group with $G=\mathbb{M}$ and point stabiliser $H$. Assume $b(G)=2$. Then one of the following holds:
\begin{itemize}\addtolength{\itemsep}{0.2\baselineskip}
\item[{\rm (i)}] Any two vertices in the Saxl graph of $G$ have a common neighbour.  
\item[{\rm (ii)}] $H$ is a $2$-local or $3$-local subgroup of $G$.
\end{itemize}
\end{proposition}

\section{Problems}\label{ss:op}

In this final section we present a number of open problems. Unless stated otherwise, $G \leqs {\rm Sym}(\O)$ is a finite transitive permutation group of degree $n$ with point stabiliser $H$ and base size $2$. 

\subsection{Connectivity}

\begin{problem}\label{pr:prim}
Prove Conjecture \ref{c:diam2} for base-two primitive groups.
\end{problem}

In Sections \ref{s:sym} and \ref{s:spor}, we verified Conjecture \ref{c:diam2} for some almost simple groups and it would be desirable to push these results further. As discussed in Remark \ref{r:imp}, different methods are needed to study $\Sigma(G)$ when $G$ has socle $T = A_m$ and $H \cap T$ acts imprimitively on $\{1, \ldots, m\}$. Similarly, for sporadic groups, it might be feasible to study some of the open cases presented in Table \ref{tab:spor} using a more sophisticated computational approach.  

One of our main tools for verifying Conjecture \ref{c:diam2} is Lemma \ref{l:q2}, which immediately gives the desired result if ${\rm val}(G)>n/2$. Therefore, it would be interesting to study the connectivity properties of the Saxl graphs of primitive groups $G$ for which this approach is not available, such as the groups highlighted in Section \ref{ss:asymp}, where ${\rm val}(G)/n \to 0$ as $|G|$ tends to infinity. In this setting, we have only verified the conjecture for the groups $G = {\rm PGL}_{2}(q)$ in Example \ref{ex:pgl2q} (and a handful of cases in Example \ref{c:pex}).

\begin{problem}\label{pr:cn}
Investigate the asymptotic behaviour of the clique number $\omega(G)$ for base-two primitive groups $G$. For example, does $\omega(G) \to \infty$ as $|G| \to \infty$?
\end{problem}

As noted in Lemma \ref{l:q2}, if $Q(G,2) \to 0$ as $|G| \to \infty$, then $\omega(G)$ does tend to infinity with $|G|$. Does this asymptotic behaviour extend to other families of primitive groups $G$? Again, it would be interesting to consider the groups $G$ with the property that $Q(G,2) \to 1$ as $|G| \to \infty$. For instance, if we take $G = {\rm PGL}_{2}(q)$ as in Example \ref{ex:pgl2q}, then $\Sigma(G)$ is the Johnson graph $J(q+1,2)$ and thus $\omega(G)=q$ does tend to infinity with $q$.

In Section \ref{s:con}, we observed that the Saxl graphs of base-two imprimitive groups can exhibit a wide range of connectivity properties, from having arbitrarily many connected components, to being connected of any given diameter. The groups we constructed in the proofs of Propositions \ref{p:c2} and \ref{p:c3} are very far from primitive, so it might be fruitful to focus on base-two imprimitive groups that are in some sense close to being primitive. For example, the \emph{quasiprimitive} groups are a natural family to consider (recall that a permutation group is quasiprimitive if every non-trivial normal subgroup is transitive). We have studied the Saxl graphs of some small simple quasiprimitive groups -- in every example, $\Sigma(G)$ is connected, but there are cases with ${\rm diam}(G)=3$, such as $(G,H) = (A_8, A_4 \times 2^2)$ and $({\rm M}_{12}, S_5)$. (Note that in the latter case, ${\rm M}_{12}$ has four conjugacy classes of subgroups isomorphic to $S_5$; for representatives of two of these classes, the corresponding Saxl graph has diameter $3$; in the other two cases, the diameter is $2$.)

\begin{problem}\label{pr:qp}
Study the connectivity properties of $\Sigma(G)$ for interesting families of base-two imprimitive groups, such as quasiprimitive groups. 
\end{problem}

\subsection{Automorphisms}

\begin{problem}\label{pr:aut}
Study the automorphism group of $\Sigma(G)$. For example, when do we have $G = {\rm Aut}(\Sigma(G))$? 
\end{problem}

Let $G \leqs {\rm Sym}(\O)$ be a primitive almost simple group with socle $T$ and observe that 
\[
G \leqs {\rm Aut}(\Sigma(G)) \leqs {\rm Sym}(\O).
\]
If we exclude a list of known exceptions (see Tables II--VI in \cite{LPS}), then any overgroup of $G$ in $\Sym(\Omega)$ not containing $\Alt(\Omega)$ also has socle $T$. For these groups, this essentially reduces Problem \ref{pr:aut} to determining whether or not all bases of size $2$ for $G$ are also bases for the almost simple overgroups of $G$ in ${\rm Sym}(\O)$.

\begin{example}
Let $G = {\rm Sp}_{2m}(2)$, where $m \geqs 6$ is even. Then $G$ has a maximal subgroup $H=S_{2m+2}$, which is embedded via the fully deleted permutation module for $H$ over $\mathbb{F}_2$ (the maximality of $H$ follows from \cite[Lemma 2.6]{BLS2}, for example, and we refer the reader to the proof of Proposition \ref{p:spb} for an explicit description of the embedding). Consider the action of $G$ on $\O = G/H$ and note that $b(G)=2$ by  \cite[Proposition 3.1]{BGS}. Let $\Sigma(G)$ be the corresponding Saxl graph. By inspecting \cite[Tables II--VI]{LPS}, we deduce that ${\rm Aut}(\Sigma(G))$ is an almost simple group with socle $G$. But ${\rm Aut}(G) = G$ and thus $G = {\rm Aut}(\Sigma(G))$.
\end{example}

\begin{example}
Consider the family of primitive groups $G \leqs {\rm Sym}(\O)$ in Example \ref{ex:johnson}. Here $G = {\rm PGL}_{2}(q)$ and $\Sigma(G)$ is the Johnson graph 
$J(q+1,2)$. Therefore 
$${\rm Aut}(\Sigma(G)) = S_{q+1}$$ 
and thus $G< {\rm Aut}(\Sigma(G))$.
Similarly, if we take the affine group $G = \mathbb{F}_{q^2}{:}D_{q+1}$ in Example \ref{ex:payley}, where $q=p^f$ for an odd prime $p$, then $\Sigma(G)$ is isomorphic to the Paley graph $P(q^2)$. Therefore  
\[
{\rm Aut}(\Sigma(G)) = \mathbb{F}_{q^2}{:}C_{(q^2-1)/2}{:}C_{2f},
\]
which is an index-two subgroup of the affine semilinear group ${\rm A\Gamma L}_{1}(q^2)$ (see \cite[Theorem 9.1]{J}). In particular, $G < {\rm Aut}(\Sigma(G))$ in this case.
\end{example}

\begin{example}
Let $L \leqs {\rm Sym}(\Delta)$ be a base-two permutation group with at least two regular suborbits. Consider the product action of $L \Wr P$ on $\O = \Delta^p$, where $P = C_p$ and $p$ is a prime. The stabiliser in $P$ of any partition of $\{1,\ldots,p\}$ into at least two parts is trivial, 
so Corollary \ref{cor:base2} implies that $b(L\Wr P)=2$. More precisely, Lemma \ref{lem:findbase} implies that $\{(\alpha_1,\ldots,\alpha_p),(\beta_1,\ldots,\beta_p)\}$ is 
an edge in $\Sigma(L\Wr P)$ if and only if each $\{\alpha_i,\beta_i\}$ is a base for $L$  and there exists $i,j$ such that $(\alpha_i,\beta_i)$ and $(\alpha_j,\beta_j)$ lie in different $L$-orbits. Therefore,
\[
L \Wr P \leqs L\Wr S_p\leqslant \Aut(\Sigma(L\Wr P)).
\]
In particular, if $p$ is odd then $L \Wr P < \Aut(\Sigma(L\Wr P))$ (even if $L = \Aut(\Sigma(L))$).
\end{example}

\begin{problem}\label{pr:rec}
To what extent does $\Sigma(G)$ determine $G$ up to permutation isomorphism? 
\end{problem}

Recall that if $G$ is a Frobenius group then $\Sigma(G)$ is complete, so we can find arbitrarily many distinct base-two groups (up to permutation isomorphism) with the same  Saxl graph. Moreover, as noted in Section \ref{s:obs}, if $L$ is a group of order $n$ acting regularly on itself then the Saxl graph for $L \times D_{2m}$ acting on $nm$ points is the complete multipartite graph with $m$ parts of size $n$. Therefore, there are arbitrarily many non-Frobenius groups with the same Saxl graph, and it would be interesting to construct further  examples. Note that this is related to Problem \ref{pr:aut}: if $G$ and $H$ are non-isomorphic groups with the same Saxl graphs, then $G$ and $H$ are subgroups of ${\rm Aut}(\Sigma(G))$ and thus $G \ne {\rm Aut}(\Sigma(G))$. 

\subsection{Cycles}

\begin{problem}\label{pr:eul}
Study the  primitive base-two groups $G$ such that $\Sigma(G)$ is non-Eulerian.  
\end{problem}

By Proposition \ref{p:eulprim}, if $\Sigma(G)$ is non-Eulerian and $G$ is insoluble, then $G \leqs L \Wr P$ in its product action, where $L$ is one of the almost simple primitive groups described in Table \ref{tab:odd} and $P \leqs S_k$ is transitive of odd order. We have found one genuine example, namely $(G,H) = ({\rm M}_{23}, 23{:}11)$. Are there any others?

\begin{problem}\label{pr:ham}
Let $G$ be a finite transitive permutation group of degree $n>2$ with connected Saxl graph $\Sigma(G)$. Prove that $\Sigma(G)$ is Hamiltonian. 
\end{problem}

As discussed in Remark \ref{r:ham}, this is a refinement of a well known conjecture of Lov\'{a}sz, which asserts that every finite connected vertex-transitive graph is Hamiltonian, except for $5$ specific graphs. We have verified the conjecture in some special cases and it would be interesting to seek further results in this direction.

\subsection{Other invariants}

\begin{problem}\label{pr:chrom}
Investigate the chromatic number of $\Sigma(G)$. 
\end{problem}

Recall that if $\Gamma$ is a finite graph, then the \emph{chromatic number} of $\Gamma$, denoted by $\chi(\Gamma)$, is the minimal number of colours needed to colour the vertices of $\Gamma$ so that no two adjacent vertices have the same colour. For any $\Gamma$, we have 
$\omega(\Gamma) \leqs \chi(\Gamma)$. Write $\chi(G)$ for the chromatic number of 
$\Sigma(G)$ and note that this is the smallest number of parts in a partition of $\Omega$ such that no part contains a base of size $2$. Since we have already noted that $\omega(G)$ can be arbitrarily large when $G$ is primitive, the same conclusion holds for $\chi(G)$.

\begin{problem}\label{pr:tdn}
Study the total domination number of $\Sigma(G)$.
\end{problem}

The \emph{total domination number} of $\Sigma(G)$, denoted by $\gamma(G)$, is the minimal size of a subset $\L$ of $\O$ with the property that for each $\a \in \O$ there exists $\b \in \L$ such that $\{\a,\b\}$ is a base for $G$ (similarly, the \emph{domination number} is the size of the smallest subset $\L$ satisfying the weaker condition that for each $\a \in \O\setminus \L$ there exists $\b \in \L$ such that $\{\a,\b\}$ is a base). Clearly, if $G$ is non-regular then $\gamma(G) \geqs 2$, and equality holds if $G$ is Frobenius (in which case, $\Sigma(G)$ is complete). We also observe that $\gamma(G)\cdot {\rm val}(G) \geqs n$. To see this, suppose $\gamma(G)=k$ and let $\Lambda$ be a total dominating set of size $k$. There are $k\cdot {\rm val}(G)$ ordered pairs $(\alpha,\beta)$ such that $\alpha\in \Lambda$ and $\beta$ is adjacent to $\alpha$. On the other hand, for each $\beta\in \Omega$, there exists $\alpha\in \Lambda$ such that $(\alpha,\beta)$ is one of these pairs, so $k\cdot {\rm val}(G)\geqs n$ as claimed.

It is natural to seek bounds on $\gamma(G)$ for some families of primitive groups, and it might also be interesting to study the groups with 
$\gamma(G)=2$. For instance, consider the action of $G = {\rm GL}_{2}(q)$ on the set 
$\O$ of nonzero vectors in $\mathbb{F}_q^2$ (see Example \ref{ex:1}). Fix a pair of vectors $u,v \in \O$ with $v \not\in \la u \ra$. Then for any $w \in \O$, either $\{u,w\}$ or $\{v,w\}$ is a basis for $\mathbb{F}_q^2$, and hence a base for $G$, so $\gamma(G)=2$.

\begin{problem}\label{pr:ind}
Study the independence number of $\Sigma(G)$.
\end{problem}

The \emph{independence number} of a graph $\Gamma$, denoted by $\a(\Gamma)$, is the maximal size of a subset of the vertices of $\Gamma$ with the property that no two vertices in the subset are joined by an edge (equivalently, $\a(\Gamma)$ is the clique number of the complement of $\Gamma$). Let us write $\a(G)$ for the independence number of $\Sigma(G)$, which  is the size of the largest subset of $\O$ such that no pair of elements forms a base. 

It is easy to see that there are imprimitive base-two groups $G$ such that $\a(G)$ is arbitrarily large. For instance, if we take $G$ to be the group constructed in the proof of Proposition \ref{p:c3}, then $G$ has blocks of imprimitivity of size $p^n$, each of which is a co-clique in $\Sigma(G)$ and thus $\a(G) \geqs p^n$. We can also find primitive groups with this property. For instance, Paley graphs are self-complementary and thus $\alpha(G)=\omega(G)$ for the group $G = V{:}D_{q+1}$ in Example \ref{ex:payley}. Moreover, \cite{BDR} gives $\a(G)=q$ and thus $\a(G)$ can be arbitrarily large. Notice that this latter family of primitive groups shows that both the clique and independence number of $\Sigma(G)$ can be arbitrarily large.  The Saxl graphs in Example \ref{ex:pgl2q} have the same property: here $G = {\rm PGL}_{2}(q)$, $\omega(G) = q$ and $\a(G) = \lfloor (q+1)/2 \rfloor$.

In order to study $\a(G)$, we introduce a new parameter $B(G)$, which may be of independent interest. A base $\Lambda$ for $G$ is \emph{minimal} if no proper subset of $\Lambda$ is a base. Let $B(G)$ be the maximal size of a minimal base. For example, $B(G)=1$ if and only if $G$ is regular, and $B(G)=2$ if $G$ is Frobenius. Clearly, we have 
\[
2^{B(G)} \leqs |G| \leqs n^{b(G)}
\]
and thus 
\begin{equation}\label{e:BG}
B(G)\leqslant b(G)\log_2 n.
\end{equation}
In \cite{Cameron}, Cameron studies a related parameter,  namely the maximum size of a minimal base over all permutation representations of a group $G$.

Following \cite{Cameron, CF}, let us say that an ordered base $(\a_1, \ldots, \a_k)$ for $G$ is \emph{irredundant} if no point is fixed by the pointwise stabiliser of its predecessors. Notice that a base is minimal if and only if it is irredundant with respect to all possible orderings. Irredundant bases are studied by Blaha \cite{blaha} (he uses the term \emph{nonredundant} base) and in the proof of \cite[Lemma 4.2]{blaha} he constructs an intransitive permutation group $G$ of degree $n$ with an irredundant base of size at least $\frac{1}{3}b(G)\log_2n$. It is straightforward to check that Blaha's irredundant base is minimal in the sense defined above, so this gives a group $G$ with 
$$B(G)\geqslant \frac{1}{3}b(G)\log_2n$$ 
and thus the upper bound in \eqref{e:BG} is essentially best possible. More generally, \cite[Theorem 2.4]{CF} states that all irredundant bases for $G$ are minimal bases if and only if they are the bases of a matroid; the permutation groups with this property are called \emph{IBIS groups}, which is an acronym for ``Irredundant Bases of Invariant Size".

Notice that if $G$ is non-regular and non-Frobenius, then   
$$\a(G) \geqs B(G) \geqs 2$$
and thus a lower bound on $B(G)$ yields a lower bound on $\a(G)$. 

\begin{remark}
It is easy to construct examples such that $\a(G)-B(G)$ is arbitrarily large. For example, take the group $G = {\rm GL}_{2}(q)$ in Example \ref{ex:1} acting on the set $\O$ of nonzero vectors in the natural module $\mathbb{F}_q^2$. Here $\{\a,\b\}$ is a base for $G$ if and only if $\a$ is not a scalar multiple of $\b$. Therefore, the set of $q-1$ scalar multiples of a fixed nonzero vector is a maximal independent set in $\Sigma(G)$ and thus $\a(G)=q-1$. However, a subset of $\O$ is a base for $G$ if and only if it contains a pair of linearly independent vectors, so $B(G)=2$.
\end{remark}

Finally, we construct a family of base-two primitive groups $G$ with the property that $B(G)$ can be arbitrarily large. 

\begin{proposition}\label{p:spb}
Let $G = {\rm Sp}_{2m}(2)$, where $m \geqs 8$ is even, and consider the action of $G$ on $\O = G/H$, where $H = S_{2m+2}$ is embedded via the fully deleted permutation module over $\mathbb{F}_2$. Then $b(G)=2$ and $B(G)\geqslant (4m+2)/5$.
\end{proposition}

\begin{proof}
To realise the embedding of $H$ in $G$, let $V$ be the $(2m+2)$-dimensional permutation module for $H$ over $\mathbb{F}_2$ and let $W$ be the $H$-invariant subspace
$$W=\{(v_1,\ldots,v_{2m+2}) \,:\, \sum_i v_i=0\}.$$
Then
$$\b((u_1,\ldots,u_{2m+2}),(v_1,\ldots,v_{2m+2}))=\sum_i u_iv_i$$
is an $H$-invariant alternating form on $W$, which induces a symplectic form $\b'$ on $W/W^\perp$, where $W^\perp=\langle (1,1,\ldots,1)\rangle$. This explains the embedding of $H$ in $G$. Moreover, $H$ is a maximal subgroup by \cite[Lemma 2.6]{BLS2} and $b(G)=2$ by \cite[Proposition 3.1]{BGS}.

Set $K={\rm Stab}_{H}(\{1,2,3,4\})\cong S_4\times S_{2m-2}$ and $e=(1,1,1,1,0,0,\ldots,0) + W^{\perp} \in W/W^{\perp}$. Then $K \leqs G_{\la e \ra}$, which is a $P_1$ parabolic subgroup of $G$. Note that $Z(G_{\la e \ra}) = \la z \ra$, where $z \in G$ is the transvection that fixes $e$ and maps $f$ to $e+f$, where $f \in W/W^{\perp}$ is a vector such that $\b'(e,f)=1$. Since $K$ is maximal in $H$ and $Z(K)=1$, it follows that $z \notin H$ and $H\cap H^z=K$. Therefore, $G$ has a suborbit $\Delta$ such that the action of $H$ on $\Delta$ is equivalent to the action of $H$ on the set of $4$-element subsets of $\{1, \ldots, 2m+2\}$. Since $2m+2 \geqs 4^2$, a result of Halasi \cite[Theorem 3.2]{halasi} implies that the latter action of $H$ has base size $\lceil  (4m+2)/5 \rceil$ and thus
$$B(G) \geqs b(H) \geqs (4m+2)/5$$
as required.
\end{proof}

\subsection{Infinite permutation groups}

\begin{problem}\label{pr:inf}
Study the Saxl graphs of infinite base-two permutation groups.
\end{problem}

Let $G$ be a group and recall that a subgroup $H$ of $G$ is \emph{malnormal} if $H \cap H^g = 1$ for all $g \in G \setminus H$. In  geometric group theory, infinite groups with non-trivial malnormal subgroups arise naturally in many different contexts (see \cite{Harpe} for a survey of  interesting examples). If $H \ne 1$ is malnormal, then we can view $G$ as a Frobenius permutation group on the set of cosets of $H$ and the corresponding Saxl graph $\Sigma(G)$ is complete. 

In a different direction, the study of bases for actions of algebraic groups was initiated by 
Burness, Guralnick and Saxl in \cite{BGSa}. Let $G$ be an algebraic group over an algebraically closed field $K$ of characteristic $p \geqs 0$, let $H$ be a closed subgroup of $G$ and consider the action of $G$ on the coset variety $\O = G/H$, which is equipped with the Zariski topology. Let us assume that the action of $G$ on $\O$ is faithful and let $b(G)$ be the base size of $G$. The \emph{generic base size} of $G$, denoted by $b^1(G)$, was introduced in \cite{BGSa}. This is defined to be the minimal $k \in \mathbb{N}$ such that the Cartesian product $\O^k$ contains a non-empty open subvariety $\Lambda$ with the property that $G_{\a_1, \ldots, \a_k}=1$ for all $(\a_1, \ldots, \a_k) \in \Lambda$. Note that $b(G) \leqs b^1(G)$.

\begin{example}
Consider the action of $G = {\rm GL}_{2}(K)$ on the set $\O$ of nonzero vectors in the natural module. Then $\{\a,\b\}$ is a base for $G$ if and only if $\a$ and $\b$ are linearly independent. Since this is an open condition (given by the non-vanishing of a determinant), it follows that $b(G)=b^1(G)=2$.
\end{example}

Write $H=G_{\a}$ and assume $b^1(G)=2$. This means that $\O$ contains a non-empty open subvariety $\Gamma$ such that $\{\a,\gamma\}$ is a base for $G$ for all $\gamma \in \Gamma$. In particular, we can associate such an open set to each vertex in $\Sigma(G)$. In the Zariski topology, any two non-empty open subsets of $\O$ have non-empty intersection and thus any two vertices in $\Sigma(G)$ have a common neighbour. In particular, $\Sigma(G)$ has diameter at most $2$. Detailed results on $b(G)$ and $b^1(G)$ are established in \cite{BGSa} in the case where $H$ is a maximal subgroup of a simple algebraic group $G$ (that is, when the action of $G$ on $\O$ is primitive). Typically, if $b(G)=2$ then we also have $b^1(G)=2$, although there are some exceptions. For the groups with $b(G)=b^1(G)=2$, we see that the algebraic group analogue of Conjecture \ref{c:diam2} holds.

By \cite[Proposition 2.5(iv)]{BGSa}, if $b(G) = 2$ then $b^1(G) \in \{2,3\}$ and it would be  interesting to focus on the corresponding version of Conjecture \ref{c:diam2} for base-two algebraic groups with $b^1(G)=3$. 
For example, \cite[Theorem 8]{BGSa} states that this situation arises when $p \ne 2$ and $H=C_G(\tau)$ for some involutory automorphism $\tau$ of $G$ that inverts a maximal torus (for instance, $(G,H) = (E_8,D_8)$ is such an example). Moreover, in these cases, each point stabiliser has a unique regular orbit on $\O$, so the problem reduces to showing that the intersection of any two of these regular orbits is non-empty. In general, it is not clear what the answer should be in this situation. 

\begin{remark}
Consider the special case where $G = {\rm PGL}_{2}(K)$ and $H=N_G(T)$ is the normaliser of a maximal torus $T$ of $G$. Here $b^1(G)=3$ and $b(G)=2$ by \cite[Theorem 9]{BGSa}. Notice that this is an infinite analogue of the set-up we studied in Example \ref{ex:johnson}. In particular, we can identify $\O=G/H$ with the variety of pairs of $1$-dimensional subspaces of $K^2$ and we see that two vertices in $\Sigma(G)$ are adjacent if and only if they have a common $1$-space. It follows that any two points in $\Sigma(G)$ have a common neighbour, so the natural analogue of Conjecture \ref{c:diam2} is true in this case. 
\end{remark}

\end{document}